\documentclass[11 pt]{article}

\usepackage{amsthm,amsmath}
\RequirePackage[numbers]{natbib}
\RequirePackage[colorlinks, citecolor=blue, urlcolor=blue]{hyperref}

\numberwithin{equation}{section}
\theoremstyle{plain}

\usepackage[]{natbib}
\usepackage{makeidx}
\usepackage{amsfonts}
\usepackage{amssymb}
\usepackage{gensymb}
\usepackage{caption}
\usepackage{bbm}
\usepackage[caption=false]{subfig}
\usepackage{lipsum}
\usepackage{xspace} 
\usepackage{enumitem}
\usepackage[english]{babel} 
\usepackage{bm} 
\usepackage{bbm} 
\usepackage[margin = 1.5in]{geometry}
\usepackage[T1]{fontenc}	 
\usepackage{lmodern}
\usepackage{amsthm} 
\usepackage{graphicx} 
\usepackage{booktabs} 
\usepackage{multicol} 
\usepackage{multirow} 
\usepackage{tocstyle} 
\usepackage[toc,page]{appendix}

\newcommand{\numbereqn}{\addtocounter{equation}{1}\tag{\theequation}} 
\usepackage{mathrsfs} 
\usepackage{lscape}
\usepackage{scalerel,stackengine} 
\stackMath
\usepackage{xcolor-patch}

\renewcommand{\epsilon}{\varepsilon}

\stackMath

\renewcommand{\epsilon}{\varepsilon}

\newtheorem{theorem}{Theorem}[section]

\newtheorem{corollary}{Corollary}[section]
\newtheorem{prop}{Proposition}[subsection]
\theoremstyle{definition}
\newtheorem{remark}{Remark}[section]

\newcommand\widecheck[1]{%
	\savestack{\tmpbox}{\stretchto{%
			\scaleto{%
				\scalerel*[\widthof{\ensuremath{#1}}]{\kern-.6pt\bigwedge\kern-.6pt}%
				{\rule[-\textheight/2]{1ex}{\textheight}}
			}{\textheight}%
		}{0.5ex}}%
	\stackon[1pt]{#1}{\scalebox{-1}{\tmpbox}}%
}

\renewcommand{\hat}{\widehat}

\newcommand{\etab}{\bm{\eta}}

\newcommand{\one}{\mathbbm{1}}

\newcommand{\js}{\text{JS}}
\newcommand{\fl}{\text{FL}}
\DeclareMathOperator{\vms}{{vM}_{s}}
\DeclareMathOperator{\vmc}{{vM}_{c}}
\DeclareMathOperator{\vm}{{vM}}
\DeclareMathOperator{\var}{{var}}

\DeclareMathOperator{\sgn}{{sgn}}

\allowdisplaybreaks

\begin{document}
	
%
		
		
		
\title{On the circular correlation coefficients for bivariate von Mises distributions on a torus \thanks{\textbf{Key words and phrases:}
		{toroidal angular models},
		{bivariate von Mises distribution},
		{circular correlation},
		{circular statistics},
		{directional data},
		{von Mises sine model},
		{von Mises cosine model.} 
		E-mail for correspondence: \texttt{samuel.wong@uwaterloo.ca}
}}

\author{Saptarshi Chakraborty$^1$ and Samuel W.K. Wong$^{2}$ \\
\\ $^1$ Department of Epidemiology \& Biostatistics, \\ Memorial Sloan-Kettering Cancer Center, New York, NY, U.S.A.
\\	$^2$ Department of Statistics and Actuarial Science, \\University of Waterloo, Waterloo, ON, Canada
}

\date{\today}
		


		
		
		

\maketitle		

\begin{abstract}
	This paper studies circular correlations for the bivariate von Mises sine and cosine distributions.  These are two simple and appealing models for bivariate angular data with five parameters each that have interpretations comparable to those in the ordinary bivariate normal model.   However, the variability and association of the angle pairs cannot be easily deduced from the model parameters unlike the bivariate normal.	Thus to compute such summary measures, tools from circular statistics are needed. We derive analytic expressions and study the properties of the Jammalamadaka-Sarma and Fisher-Lee circular correlation coefficients for the von Mises sine and cosine models.  Likelihood-based inference of these coefficients from sample data is then presented.  
	The correlation coefficients are illustrated with numerical and visual examples, and the maximum likelihood estimators are assessed on simulated and real data, with comparisons to their non-parametric counterparts.  Implementations of these computations for practical use are provided in our  R package \texttt{BAMBI}.
\end{abstract}

	\section{Introduction}
	
	Directional (or circular) statistics concerns the modeling and analysis of data that can be expressed via directions (unit vectors), axes or rotations.  In the case where data are angles in $[-\pi, \pi)$ on a circle (univariate) or a toroid (multivariate), they are called \emph{angular data}.  
	This paper focuses on assessing the correlation in bivariate angular data, where the coordinates have support  $[-\pi, \pi)^2$.  Since angles have an inherent wraparound nature, specialized distributions and descriptive statistics have been developed for their study (see, e.g.,  \cite{mardia:jupp:2009} for a review).  We begin with a review of the relevant concepts that motivate our work.
	
	
	\subsection{Univariate and bivariate von Mises distributions}
	
	The von Mises distribution is perhaps the most well-known univariate circular distribution, because of its ease of use and close relationship with the normal distribution (see, e.g., \citep{fisher:1995, mardia:jupp:2009}). Formally, an angular random variable $\Theta$ with support $[-\pi, \pi)$ (or any other interval of length $2\pi$) is said to follow the von Mises distribution with parameters $\mu_1 \in [-\pi, \pi)$ and $\kappa_1 \geq 0$, if it has density 
	\[
	f_{\vm}(\theta) = (2\pi I_0 (\kappa_1))^{-1} \exp\{\kappa_1 \cos(\theta - \mu_1) \}
	\]
	where $I_r(\cdot)$ denotes the modified Bessel function of the first kind with order $r$. A multivariate generalization for this distribution is however not straightforward and there is no unique way of defining a  multivariate distribution with univariate von Mises-like marginals and conditionals.  In the bivariate case (i.e., paired angles on a torus) two versions have been suggested for practical use, namely the sine model (\citet{singh:2002}) and the cosine model (\citet{mardia:2007}); both have found important applications in protein bioinformatics (\citep{lennox:2009, bhattacharya:2015}). Formal definitions of the two distributions are as follows.  Given the parameters  $\mu_1, \mu_2 \in [-\pi, \pi)$, $\kappa_1, \kappa_2 \geq 0$ and $\kappa_3 \in (-\infty, \infty)$ a pair of angular random variables $(\Theta, \Phi)$ with support $[-\pi, \pi)^2$ is said to follow the (bivariate) von Mises sine distribution, denoted $(\Theta, \Phi) \sim \vms(\mu_1, \mu_2, \kappa_1, \kappa_2, \kappa_3)$, if the pair has joint density
	\begin{align*} \label{vms}
	f_{\vms}(\theta, \phi) = { C_s(\kappa_1, \kappa_2, \kappa_3)}^{-1} \exp [ & \kappa_1 \cos(\theta - \mu_1) + \kappa_2 \cos(\phi - \mu_2) + \\ 
	& \quad  \kappa_3 \sin(\theta - \mu_1) \sin(\phi - \mu_2) ] \numbereqn
	\end{align*}
	where the reciprocal of the normalizing constant is given by
	\begin{align} \label{c_vms}
	C_s(\kappa_1, \kappa_2, \kappa_3) = 4 \pi^2 \sum_{m = 0} ^\infty  {{2m}\choose{m}} \left(\frac{\kappa_3^2}{4 \kappa_1 \kappa_2}\right)^m I_m(\kappa_1) I_m(\kappa_2).
	\end{align}
	In contrast, the pair is said to follow the von Mises cosine distribution,
	denoted $(\Theta, \Phi) \sim \vmc(\mu_1, \mu_2, \kappa_1, \kappa_2, \kappa_3)$, if the pair has joint density
	\begin{align*} \label{vmc}
	f_{\vmc}(\theta, \phi) =  C_c(\kappa_1, \kappa_2, \kappa_3)^{-1} \exp [ & \kappa_1 \cos(\theta - \mu_1) + \kappa_2 \cos(\phi - \mu_2) + \\
	& \quad \kappa_3 \cos(\theta - \mu_1 - \phi + \mu_2) ] \numbereqn
	\end{align*}
	where  the reciprocal of the normalizing constant is
	given by
	\begin{align} \label{c_vmc}
	C_c(\kappa_1, \kappa_2, \kappa_3) = 4 \pi^2 \left\lbrace I_0(\kappa_1) I_0(\kappa_2) I_0(\kappa_3) + 2 \sum_{m = 0} ^\infty I_m(\kappa_1) I_m(\kappa_2) I_m(\kappa_3)  \right \rbrace.
	\end{align}
	Note that both densities reduce to the Uniform density over $[-\pi, \pi]^2$ when $\kappa_1$, $\kappa_2$ and $\kappa_3$ are all zero, which is analogous to the univariate von Mises circular density. 
	
	Although other generalizations with more parameters have been studied theoretically (see, e.g., \citet{mardia:1975, rivest:1988}), the von Mises sine and cosine distributions are appealing because of their simplicity and ease of use. Moreover, both the models have close relationships with the bivariate normal distribution on $\mathbb{R}^2$. First, both the models have five parameters, with comparable interpretations to those in the  bivariate normal. Second, under certain conditions both densities closely approximate the bivariate normal density. Due to symmetry of the corresponding marginal distributions (see \citep{mardia:2007, singh:2002}), it immediately follows that $\mu_1$ and $\mu_2$ are the respective circular means in both the sine and cosine models (see, e.g., \citep{mardia:jupp:2009, jammalamadaka:2001} for the definition of circular mean). The parameters $\kappa_1$ and $\kappa_2$ are  the so-called ``concentration'' (or ``anti-variance'') parameters, and $\kappa_3$ is the ``covariance'' (or ``correlation'') parameter \citep{mardia:2007}, and together they describe the concentrations (or precisions) and dependence between the random coordinates $\Theta$ and $\Phi$. As in the bivariate normal model, a necessary and sufficient condition for $\Theta$ and $\Phi$ to be independent is given by $\kappa_3 = 0$.
	
	It is to be noted however, that $\kappa_1$, $\kappa_2$ and $\kappa_3$ need to be reported together to describe the variability and association between $\Theta$ and $\Phi$ -- the parameters $\kappa_1$ and $\kappa_2$ alone do not characterize the variances and $\kappa_3$ alone does not explain the association between $\Theta$ and $\Phi$. Moreover, the variances and association depend on these parameters through complicated functions which 	are difficult to visualize, and they cannot be well approximated by any simple functions of the three parameters in general. This is a key distinction from a bivariate normal model. There is also no requirement that the square of the ``covariance'' parameter ($\kappa_3$) be bounded above by the product of the ``concentration'' parameters ($\kappa_1$ and $\kappa_2$). This flexibility permits bimodality in the sine model density when $\kappa_3^2 > \kappa_1 \kappa_2$, and in the cosine model density when $\kappa_3 < - \kappa_1 \kappa_2/(\kappa_1+\kappa_2)$ (see \citep{mardia:2007}).  Specifically, we are interested in association for the circular context, which we define next.
	
	
	\subsection{Circular correlation coefficients}
	
	To describe the association between an angle pair, we may use circular correlation coefficients. 
	Different parametric circular correlation coefficients have been proposed in the literature. In this paper we consider the Jammalamadaka-Sarma coefficient (\citet{jammalamadaka:1988}) and the Fisher-Lee coefficient (\citet{fisher:1983}), which were designed with analogy to the ordinary correlation coefficient.\footnote{In the literature the JS and FL correlation coefficients have been denoted by $\rho_c$ and $\rho_T$ respectively, following the authors' notations. We however, shall use $\rho_{\fl}$ and $\rho_{\js}$ in this paper for clarity.} Formal definitions of these coefficients for a pair of random toroidal angles $(\Theta, \Phi)$ are given as follows. Let $\mu_1$ and $\mu_2$ be the circular means of $\Theta$ and $\Phi$ respectively. Then the Jammalamadaka-Sarma (JS) circular correlation coefficient is defined as
	\begin{equation} \label{rho_js_defn}
	\rho_{\js}(\Theta, \Phi) = \frac{E\left[\sin (\Theta - \mu_1) \sin (\Phi - \mu_2) \right]}{\sqrt{E\left[\sin^2 (\Theta - \mu_1) \right] E\left[\sin^2  (\Phi - \mu_2)\right]}}.
	\end{equation} 
    Now let $(\Theta_1, \Phi_1)$ and $(\Theta_2, \Phi_2)$ be i.i.d. copies of $(\Theta, \Phi)$. Then the Fisher-Lee (FL) circular correlation coefficient is defined by
	\begin{equation} \label{rho_fl_defn}
	\rho_{\fl}(\Theta, \Phi) = \frac{E\left[\sin (\Theta_1 - \Theta_2) \sin (\Phi_1 - \Phi_2) \right]}{\sqrt{E\left[\sin^2 (\Theta_1 - \Theta_2) \right] E\left[\sin^2  (\Phi_1 - \Phi_2)\right]}}.
	\end{equation}
	Observe that $\rho_{\js}$ resembles the standard form of the usual Pearson product moment correlation coefficient, while $\rho_{\fl}$ is analogous to its U-statistic form. Both $\rho_{\js}$ and $\rho_{\fl}$ possess properties similar to the ordinary correlation coefficient. In particular, $\rho_{\js}, \rho_{\fl} \in [-1, 1]$ and they are equal to 1 (-1) under perfect positive (negative) toroidal-linear (\emph{T-linear}) relationship \citep{fisher:1983,jammalamadaka:1988}. Moreover, under independence, they are both equal to zero, although the reverse implication is not necessarily true.
	
	In practice the sample versions of these coefficients, obtained by replacing the expectations by sample averages and the circular mean parameters ($\mu_1$ and $\mu_2$ in $\rho_{\js}$) by the sample circular means, can be used for estimation of the population coefficients. The resulting estimates are called the non-parametric (or distribution-free) estimates of $\rho_{\js}$ and $\rho_{\fl}$ in this paper.  Approximate confidence intervals for the population coefficients can then be constructed using the asymptotic normal distributions of these non-parametric estimates \citep{fisher:1983,jammalamadaka:1988}. 
	
	\subsection{Contribution of this work}
	
	Suppose the von Mises sine or cosine distribution is used to model bivariate angular data.  This leads to two natural questions that form the basis of our contributions in this paper:  (a) what are the properties of $\rho_{\js}$ and $\rho_{\fl}$ in these bivariate von Mises distributions? and  (b) how does likelihood-based inference on $\rho_{\js}$ and $\rho_{\fl}$ perform when these distributions are fitted to sample data?

	The first question is addressed in Section~\ref{sec_thm_cor}.  We note that analytic expressions for the true or population versions of $\rho_{\js}$ and $\rho_{\fl}$ need to be derived for specific distributions. In their original papers \citep{jammalamadaka:1988, fisher:1983}, such expressions for $\rho_{\js}$ and $\rho_{\fl}$ were given for the bivariate wrapped normal distribution (which is the \emph{wrapped} version of the bivariate normal distribution on $[-\pi, \pi)$). However, the correlation coefficients for the bivariate von Mises distributions have not been previously studied. As a theoretical contribution of this paper, we derive expressions for $\rho_{\fl}$ and $\rho_{\js}$ for both von Mises sine and cosine distributions, and discuss additional properties of these correlation coefficients that follow from these expressions. The analytic forms of these expressions involve singly infinite series containing Bessel functions, and can be easily computed in any standard statistical software. 	As a practical consequence, we have incorporated fast and accurate evaluations of these theoretical quantities in our R package \texttt{BAMBI} \citep{bambi_rpack}.  
	We conclude the section with numerical and visual illustrations that provide insights into the behavior and interpretability of the two correlation coefficients.
	
	We turn to the second question in Section \ref{sec:inference}.  We consider the practical situation where the von Mises sine or cosine models are fitted to sample data via the method of maximum likelihood.  Then, likelihood theory provides corresponding estimates of $\rho_{\js}$ and $\rho_{\fl}$ based on these fitted parametric distributions, along with asymptotic normal confidence intervals.  We present statistical inference for these coefficients, and illustrate their application on simulated and real data examples.  We also compare the performance of the likelihood-based estimates to the non-parametric counterparts  \citep{fisher:1983,jammalamadaka:1988}.
	
	
	
	\section{Circular correlation coefficients for bivariate von Mises sine and cosine distributions} \label{sec_thm_cor}
	
	In this section, we first present theoretical properties associated with the circular correlation coefficients for bivariate von Mises sine and cosine distributions.  Following, we numerically and visually illustrate these properties on examples of these distributions.
	
	\subsection{Theoretical properties}
	
	We begin by deriving analytic expressions for $\rho_{\js}$ and $\rho_{\fl}$ for both von Mises sine and cosine models in Theorem~\ref{thm_vms} and \ref{thm_vmc}. The expressions for  the marginal circular variances are also provided for completeness.\footnote{Circular variance is defined for an angular variable $\Theta$ as $\var(\Theta) = 1-E(\cos(\Theta))$ (see, e.g.,~\cite{jammalamadaka:2001}).  Expressions for $\var(\Theta)$ and $\var(\Phi)$ for the von Mises sine distribution were first provided in \citet{singh:2002}.} The proofs are based on a series of technical results, which we state and prove in Propositions~\ref{prop_vms_results}-\ref{prop_vmc_results} in Appendix~\ref{appen_tech_results}. We then establish connections between the signs of the ``covariance'' parameters in the original models, and the sign of the corresponding $\rho_{\js}$ (Corollary~\ref{cor_js_sign}). This is followed by Corollary~\ref{cor_fl_js_reln} exhibiting a mathematical relationship between $\rho_{\fl}$ and $\rho_{\js}$. We make a few  interesting remarks which provide deeper insights to the behaviors of the two correlation coefficients under the sine and cosine models. Finally, in Corollaries~\ref{cor_circ_lin_reln_vmsin} and \ref{cor_circ_lin_reln_vmcos}, we provide sufficient conditions for  $\rho_{\fl}$ and $\rho_{\js}$ to be approximately equal, and well-approximated by simple closed form functions of $\kappa_1$, $\kappa_2$ and $\kappa_3$.

	\begin{theorem} \label{thm_vms}
		Let $(\Theta, \Phi)$ have a joint bivariate von Mises sine distribution \citep{singh:2002} with parameters $\kappa_1, \kappa_2, \kappa_3, \mu_1$, and $\mu_2$. 
		\begin{enumerate}
			\item If $(\Theta_1, \Phi_1)$ and $(\Theta_2, \Phi_2)$ denotes two i.i.d. copies of $(\Theta, \Phi)$, then the Fisher-Lee circular correlation coefficient (\ref{rho_fl_defn}) between $\Theta$ and $\Phi$ is given by 
			\begin{align*}
			\rho_{\fl} (\Theta, \Phi) 
			&= \frac{\left(\frac{1}{C_s} \frac{\partial C_s}{\partial \kappa_3} \right) \left(\frac{1}{C_s} \frac{\partial^2 C_s}{\partial \kappa_1 \partial \kappa_2} \right)}{\sqrt{\left(\frac{1}{C_s} \frac{\partial^2 C_s}{\partial \kappa_1^2} \right) \left(1 - \frac{1}{C_s} \frac{\partial^2 C_s}{\partial \kappa_1^2} \right) \left(\frac{1}{C_s} \frac{\partial^2 C_s}{\partial \kappa_2^2} \right) \left(1 - \frac{1}{C_s} \frac{\partial^2 C_s}{\partial \kappa_2^2} \right)}}.  \numbereqn
			\end{align*}
			
			\item The Jammalamadaka-Sarma circular correlation coefficient (\ref{rho_js_defn}) between $\Theta$ and $\Phi$ is given by 
			\begin{align*}
			\rho_{\js} (\Theta, \Phi) 
			&= \frac{\frac{1}{C_s} \frac{\partial C_s}{\partial \kappa_3} }{\sqrt{\left(1 - \frac{1}{C_s} \frac{\partial^2 C_s}{\partial \kappa_1^2} \right)  \left(1 - \frac{1}{C_s} \frac{\partial^2 C_s}{\partial \kappa_2^2} \right)}}. \numbereqn 
			\end{align*}
			
			\item The circular variances for $\Theta$ and $\Phi$ are given by \citep{singh:2002}
			\begin{align*}
			\var(\Theta) = 1 - \frac{1}{C_s} \frac{\partial C_s}{\partial \kappa_1} \text{ and } \var(\Phi) = 1 - \frac{1}{C_s} \frac{\partial C_s}{\partial \kappa_2}.
			\end{align*}
		\end{enumerate}
		
		\noindent Here  $C_s$ denotes the reciprocal of the normalizing constant as given in (\ref{c_vms}), and infinite series representations of the partial derivatives of $C_s$ are given in Remark~\ref{del_C_expr}. 
	\end{theorem}
	
		\begin{proof}
		Without loss of generality, assume that $\mu_1 = \mu_2 = 0$.
		\begin{enumerate}
			\item 	
			Note that, because $(\Theta_1, \Phi_1)$ $(\Theta_2, \Phi_2)$ are i.i.d.,
			\begin{align*}
			E\left[\sin (\Theta_1 - \Theta_2) \sin (\Phi_1 - \Phi_2) \right] 
			&\stackrel{(a)}{=} 2\: E \left(\sin \Theta_1 \sin \Phi_1 \right) E\left(\cos \Theta_1 \cos \Phi_1\right) \\ 
			&\stackrel{(b)}{=} 2\: \left(\frac{1}{C_s} \frac{\partial C_s}{\partial \kappa_3} \right) \left(\frac{1}{C_s} \frac{\partial^2 C_s}{\partial \kappa_1 \partial \kappa_2} \right)
			\end{align*}
			where $(a)$ follows from Proposition~\ref{vms_expec_res}~\ref{E_sin_x_cos_y} and  $(b)$ follows from Proposition~\ref{vms_expec_res}~\ref{E_sin_x_sin_y} and \ref{E_cos_x_cos_y}. Moreover,
			\begin{align*}
			E\left[\sin^2 \left(\Theta_1 - \Theta_2\right)\right] 
			&\stackrel{(c)}{=} 2 \: E \left(\cos^2 \Theta_1 \right) E \left(\sin^2 \Theta_1 \right) \\ 
			&= 2 \: E \left(\cos^2 \Theta_1 \right) \left[1 - E \left(\cos^2 \Theta_1 \right)\right] \\
			&\stackrel{(d)}{=} 2 \left(\frac{1}{C_s} \frac{\partial^2 C_s}{\partial \kappa_1^2} \right) \left(1 - \frac{1}{C_s} \frac{\partial^2 C_s}{\partial \kappa_1^2} \right) 
            \end{align*}
			with $(c)$ and $(d)$ following from Propositions~\ref{vms_expec_res}~\ref{E_sin_x_cos_x} and \ref{vms_expec_res}~\ref{E_cos2_x} respectively. Similarly
			\[
			E\left[\sin^2 \left(\Phi_1 - \Phi_2\right)\right] =  \left(\frac{1}{C_s} \frac{\partial^2 C_s}{\partial \kappa_2^2} \right) \left(1 - \frac{1}{C_s} \frac{\partial^2 C_s}{\partial \kappa_2^2} \right).
			\]
			This completes the proof.

			\item From Proposition~\ref{vms_expec_res}~\ref{E_sin_x_sin_y} and  Proposition~\ref{vms_expec_res}~\ref{E_cos2_x}
			\begin{gather*}
			E\left(\sin \Theta \sin \Phi \right) = \frac{1}{C_s} \frac{\partial C_s}{\partial \kappa_3} \\
			E \left(\sin^2 \Theta \right) = 1 - \frac{1}{C_s} \frac{\partial^2 C_s}{\partial \kappa_1^2} \\ \text{ and } E \left(\sin^2 \Phi \right) = 1 - \frac{1}{C_s} \frac{\partial^2 C_s}{\partial \kappa_2^2}.
			\end{gather*}
			The proof is completed by plugging these expressions into the formula.
			
			\item This is proved in \citet{singh:2002} (Proposition~\ref{vms_expec_res}~\ref{E_cos_x}).
		\end{enumerate}	
	\end{proof}

	\begin{theorem} \label{thm_vmc}
		Let $(\Theta, \Phi)$ have a joint bivariate von Mises cosine distribution \citep{mardia:2007} with parameters $\kappa_1, \kappa_2, \kappa_3, \mu_1$, and $\mu_2$. 
		\begin{enumerate}
			\item If $(\Theta_1, \Phi_1)$ and $(\Theta_2, \Phi_2)$ denotes two i.i.d. copies of $(\Theta, \Phi)$, then the Fisher-Lee circular correlation coefficient (\ref{rho_fl_defn}) between $\Theta$ and $\Phi$ is given by 
			\begin{align*}
			\rho_{\fl} (\Theta, \Phi) 
			&= \frac{\left(\frac{1}{C_c} \left\{\frac{\partial C_c}{\partial \kappa_3} -  \frac{\partial^2 C_c}{\partial \kappa_1 \partial \kappa_2} \right\} \right) \left(\frac{1}{C_c} \frac{\partial^2 C_c}{\partial \kappa_1 \partial \kappa_2} \right)}{\sqrt{\left(\frac{1}{C_c} \frac{\partial^2 C_c}{\partial \kappa_1^2} \right) \left(1 - \frac{1}{C_c} \frac{\partial^2 C_c}{\partial \kappa_1^2} \right) \left(\frac{1}{C_c} \frac{\partial^2 C_c}{\partial \kappa_2^2} \right) \left(1 - \frac{1}{C_c} \frac{\partial^2 C_c}{\partial \kappa_2^2} \right)}}.  \numbereqn
			\end{align*}
			
			\item The Jammalamadaka-Sarma circular correlation coefficient (\ref{rho_js_defn}) between $\Theta$ and $\Phi$ is given by 
			\begin{align*}
			\rho_{\js} (\Theta, \Phi) 
			&= \frac{\frac{1}{C_c} \left\{\frac{\partial C_c}{\partial \kappa_3} -  \frac{\partial^2 C_c}{\partial \kappa_1 \partial \kappa_2} \right\} }{\sqrt{\left(1 - \frac{1}{C_c} \frac{\partial^2 C_c}{\partial \kappa_1^2} \right)  \left(1 - \frac{1}{C_c} \frac{\partial^2 C_c}{\partial \kappa_2^2} \right)}}. \numbereqn 
			\end{align*}
			
			\item The circular variances for $\Theta$ and $\Phi$ are given by
			\begin{align*}
			\var(\Theta) = 1 - \frac{1}{C_c} \frac{\partial C_c}{\partial \kappa_1} \text{ and } \var(\Phi) = 1 - \frac{1}{C_c} \frac{\partial C_c}{\partial \kappa_2}.
			\end{align*}
			
		\end{enumerate}
		
		\noindent Here  $C_c$ denotes the reciprocal of the normalizing constant as given in (\ref{c_vmc}), and infinite series representations for the partial derivatives of $C_c$ are given in Remark~\ref{del_C_expr}. 
	\end{theorem}

    \begin{proof}
    	This proof closely resembles the proof of Theorem~\ref{thm_vms} for the most part. Without loss of generality, assume that $\mu_1 = \mu_2 = 0$.
    	\begin{enumerate}
    		\item 	
    		Note that, because $(\Theta_1, \Phi_1)$ $(\Theta_2, \Phi_2)$ are i.i.d.,
    		\begin{align*}
    		& \quad 	E\left[\sin (\Theta_1 - \Theta_2) \sin (\Phi_1 - \Phi_2) \right] \\
    		&\stackrel{(a)}{=} 2\: E \left(\sin \Theta_1 \sin \Phi_1 \right) E\left(\cos \Theta_1 \cos \Phi_1\right) \\
    		&\stackrel{(b)}{=} 2\: \left(\frac{1}{C_c} \left\{\frac{\partial C_c}{\partial \kappa_3} -  \frac{\partial^2 C_c}{\partial \kappa_1 \partial \kappa_2} \right\} \right) \left(\frac{1}{C_c} \frac{\partial^2 C_c}{\partial \kappa_1 \partial \kappa_2} \right) 
    		\end{align*}
    		with $(a)$ being a consequence of Proposition~\ref{vmc_expec_res}~\ref{E_c_sin_x_cos_y} and $(b)$ of Proposition~\ref{vmc_expec_res}~\ref{E_c_sin_x_sin_y} and \ref{E_c_cos_x_cos_y}. Also,
    		\begin{align*}
    	    E\left[\sin^2 \left(\Theta_1 - \Theta_2\right)\right] 
    		&\stackrel{(c)}{=} 2 \: E \left(\cos^2 \Theta_1 \right) E \left(\sin^2 \Theta_1 \right) \\ 
    		&= 2 \: E \left(\cos^2 \Theta_1 \right) \left[1 - E \left(\cos^2 \Theta_1 \right)\right] \\
    		&\stackrel{(d)}{=} 2 \left(\frac{1}{C_c} \frac{\partial^2 C_c}{\partial \kappa_1^2} \right) \left(1 - \frac{1}{C_c} \frac{\partial^2 C_c}{\partial \kappa_1^2} \right) 
    		\end{align*}
    		where $(c)$ and $(d)$  follows from Propositions~\ref{vmc_expec_res}~\ref{E_c_sin_x_cos_x} and \ref{vmc_expec_res}~\ref{E_c_cos2_x} respectively. Similarly,
    		\[
    		E\left[\sin^2 \left(\Phi_1 - \Phi_2\right)\right] =  \left(\frac{1}{C_c} \frac{\partial^2 C_c}{\partial \kappa_2^2} \right) \left(1 - \frac{1}{C_c} \frac{\partial^2 C_c}{\partial \kappa_2^2} \right).
    		\]
    		This completes the proof.

    		\item From Proposition~\ref{vmc_expec_res}~\ref{E_c_sin_x_sin_y} and  Proposition~\ref{vmc_expec_res}~\ref{E_c_cos2_x}
    		\begin{gather*}
    		E\left[\sin \Theta \sin \Phi \right] = \frac{1}{C_c} \frac{\partial C_c}{\partial \kappa_3} \\
    		E \left(\sin^2 \Theta \right) = 1 - \frac{1}{C_c} \frac{\partial^2 C_c}{\partial \kappa_1^2} \\ \text{ and } E \left(\sin^2 \Phi \right) = 1 - \frac{1}{C_c} \frac{\partial^2 C_c}{\partial \kappa_2^2}.
    		\end{gather*}
    		
    		The proof is completed by plugging these expressions into the formula.
    		
    		\item Follows from Proposition~\ref{vmc_expec_res}~\ref{E_c_cos_x} and the definition of circular variance.
    		
    	\end{enumerate}	
    	
    \end{proof}

		\begin{remark} \label{del_C_expr}
		Expressions for the partial derivatives for the von Mises sine and cosine normalizing constants as infinite series are provided below. The derivations are straightforward using formulas in \citet[\S 9.6]{abramowitz:1964} and are therefore, omitted. 
		
		\begin{enumerate}
			\item Sine model:
			\begin{align}
			\frac{\partial C_s}{\partial \kappa_1} &= 4 \pi^2  \sum_{m=0}^{\infty} \binom{2m}{m} \left(\frac{\kappa_3^2}{4\kappa_1 \kappa_2}\right)^m I_{m+1}(\kappa_1) I_m(\kappa_2) \label{del_C_k1_expr} \\
			\frac{\partial C_s}{\partial \kappa_2} &= 4 \pi^2  \sum_{m=0}^{\infty} \binom{2m}{m} \left(\frac{\kappa_3^2}{4\kappa_1 \kappa_2}\right)^m I_{m}(\kappa_1) I_{m+1}(\kappa_2) \label{del_C_k2_expr} \\
			\frac{\partial C_s}{\partial \kappa_3} &=  8 \pi^2  \sum_{m=1}^{\infty} m \binom{2m}{m} \frac{\kappa_3^{2m-1}}{(4\kappa_1 \kappa_2)^m} I_{m}(\kappa_1) I_{m}(\kappa_2) \label{del_C_lambda_expr} \\
			\frac{\partial^2 C_s}{\partial \kappa_1^2} &= 4 \pi^2  \sum_{m=0}^{\infty} \binom{2m}{m} \left(\frac{\kappa_3^2}{4\kappa_1 \kappa_2}\right)^m  \nonumber  \\
			&\qquad \qquad \left(\frac{I_{m+1}(\kappa_1)}{\kappa_1} + I_{m+2}(\kappa_1)\right) I_m(\kappa_2) \label{del_C_k1_k1_expr} \\
			\frac{\partial^2 C_s}{\partial \kappa_2^2} &= 4 \pi^2  \sum_{m=0}^{\infty} \binom{2m}{m} \left(\frac{\kappa_3^2}{4\kappa_1 \kappa_2}\right)^m \nonumber \\
			& \qquad \qquad  I_m(\kappa_1) \left(\frac{I_{m+1}(\kappa_2)}{\kappa_2} + I_{m+2}(\kappa_2)\right) \label{del_C_k2_k2_expr} \\
			\frac{\partial^2 C_s}{\partial \kappa_1 \: \partial \kappa_2} &= 4 \pi^2  \sum_{m=0}^{\infty} \binom{2m}{m} \left(\frac{\kappa_3^2}{4\kappa_1 \kappa_2}\right)^m  I_{m+1}(\kappa_1) I_{m+1}(\kappa_2) \label{del_C_k1_k2_expr}
			\end{align}
			
			Note that, repeated applications of L'Hospital's rule on the relationship $\frac{\partial I_m (x)}{\partial x} = \frac{1}{2} I_{m-1}(x) + \frac{1}{2} I_{m+1}(x)$ yields $\lim\limits_{x \rightarrow 0} \frac{I_m(x)}{x^m} = 2^{-m}$ for any integer $m \geq 0$ and $\lim\limits_{x \rightarrow 0} \frac{I_n(x)}{x^m} = 0$ for integers $n > m \geq 0$. Thus, when $\kappa_1$ and/or $\kappa_2$ is zero the above expressions remain valid, and can be further simplified.
			
			\item Cosine model:
			\begin{align}
			\frac{\partial C_c}{\partial \kappa_1} =  4\pi^2  & \left\{ I_1(\kappa_1) I_0(\kappa_2) I_0(\kappa_3)  +  \right. \nonumber \\
			&  \left. \qquad  \sum_{m = 1} ^\infty  I_m(\kappa_2) I_m(\kappa_3) \left[ I_{m+1}(\kappa_1) + I_{m-1}(\kappa_1) \right] \right\} \label{del_C_c_k1_expr} \\
			\frac{\partial C_c}{\partial \kappa_2} = 4\pi^2 & \left\{  I_0(\kappa_1) I_1(\kappa_2) I_0(\kappa_3)  +  \right. \nonumber \\
			&  \left. \qquad  \sum_{m = 1} ^\infty  I_m(\kappa_1) I_m(\kappa_3) \left[ I_{m+1}(\kappa_2) + I_{m-1}(\kappa_2) \right] \right\} \label{del_C_c_k2_expr} \\
			\frac{\partial C_c}{\partial \kappa_3} = 4\pi^2 & \left\{ I_0(\kappa_1) I_0(\kappa_2) I_1(\kappa_3)  +  \right. \nonumber \\ 
			&  \left. \qquad \sum_{m = 1} ^\infty  I_m(\kappa_1) I_m(\kappa_2) \left[ I_{m+1}(\kappa_3) + I_{m-1}(\kappa_3) \right] \right \rbrace.  \label{del_C_c_k3_expr}\\
			\frac{\partial^2 C_c}{\partial \kappa_1^2} = 2\pi^2 & \left\{  I_0(\kappa_2) I_0(\kappa_3)[I_0(\kappa_1) + I_2(\kappa_1)] + \right. \nonumber \\ 
			&  \left. \qquad  \sum_{m = 1} ^\infty I_{m}(\kappa_2) I_{m}(\kappa_3) [I_{m-2}(\kappa_1) + 2I_{m}(\kappa_1) + I_{m+2}(\kappa_1)]   \right\} \label{del_C_c_k1_k1_expr}\\ 
			\frac{\partial^2 C_c}{\partial \kappa_2^2} = 2\pi^2 & \left\{  I_0(\kappa_1) I_0(\kappa_3)[I_0(\kappa_2) + I_2(\kappa_2)] + \right. \nonumber \\ 
			&  \left. \qquad  \sum_{m = 1} ^\infty I_{m}(\kappa_1) I_{m}(\kappa_3) [I_{m-2}(\kappa_2) + 2I_{m}(\kappa_2) + I_{m+2}(\kappa_2)]   \right\} \label{del_C_c_k2_k2_expr}\\
			\frac{\partial^2 C_c}{\partial \kappa_1 \partial \kappa_2} = 2\pi^2 & \left\{ 2 I_1(\kappa_1) I_1(\kappa_2) I_0(\kappa_3) + \right. \nonumber \\ 
			&  \left. \sum_{m = 1} ^\infty  I_m(\kappa_3) \left[ I_{m+1}(\kappa_1) + I_{m-1}(\kappa_1) \right] \left[ I_{m+1}(\kappa_2) + I_{m-1}(\kappa_2) \right]\right\} \label{del_C_c_k1_k2_expr}
			\end{align}
			
		\end{enumerate}
	\end{remark}
	
	\begin{corollary} \label{cor_js_sign}
	$\sgn(\rho_{\js}) = \sgn(\kappa_3)$, where $\sgn (x)$ is the sign of a real number $x$ defined as $\sgn (x) = \one_{(0, \infty)} (x) - \one_{(-\infty, 0)} (x)$. In other words, the direction (sign) of the T-linear association between the two  coordinates in von Mises sine and cosine distributions, as depicted by the JS circular correlation coefficient, are determined by the sign of the associated ``covariance'' parameter.
	\end{corollary}

    \begin{proof}
	  These results follow from Proposition~\ref{vms_expec_res}~\ref{E_sin_x_sin_y_sgn} and Proposition~\ref{vmc_expec_res}~\ref{E_c_sin_x_sin_y_sgn} and the definition of $\rho_{\js}$.
   \end{proof}

	 \begin{remark} \label{rem_vms_indep}
		As immediate consequences of Corollary~\ref{cor_js_sign}, it follows that $\rho_{\js} \lesseqgtr 0$ if and only if $\kappa_3 \lesseqgtr 0$. This in particular means $\rho_{\js} = 0$ implies $\kappa_3 = 0$, which in turn characterizes independence in the respective models. Thus, for both von Mises sine and cosine models \emph{uncorrelatedness (in the sense of \citet{jammalamadaka:1988}) implies independence}, which is analogous to a bivariate normal distribution. 
	\end{remark}

	\begin{corollary} \label{cor_fl_js_reln}
		For both von Mises sine and cosine models, 
		\[
		\rho_{\fl}= \delta\left(\cos (\Theta-\mu_1), \cos (\Phi-\mu_2)\right) \rho_{\js} 
		\]
		where $\delta(X, Y) = E(XY)/\sqrt{E(X^2)E(Y^2)}$. This, in particular, implies (via the Schwarz inequality) that $|\rho_{\fl}| \leq |\rho_{\js}|$  for both sine and cosine models.
	\end{corollary}

       \begin{proof} 
    	This is an immediate consequence of Theorems~\ref{thm_vms} and \ref{thm_vmc}, and the facts that $\frac{\partial^2 C_\alpha}{\partial \kappa_1 \partial \kappa_2} = E[\cos (\Theta-\mu_1) \cos (\Phi-\mu_2)]$, $\frac{\partial^2 C_\alpha}{\partial \kappa_1^2} = E[\cos^2(\Theta - \mu_1)]$ and $\frac{\partial^2 C_\alpha}{\partial \kappa_2^2} = E[\cos^2(\Phi - \mu_2)]$  ($\alpha = s, c$) for both von Mises sine and cosine models (see Proposition~\ref{vms_expec_res}~\ref{E_cos_x_cos_y},\ref{E_cos2_x} and \ref{vmc_expec_res}~\ref{E_c_cos_x_cos_y},\ref{E_c_cos2_x}).
    \end{proof}

	\begin{remark}\label{rem_fl_js_sign}
	 From Corollary~\ref{cor_fl_js_reln} it follows that $\rho_{\js}$ and $\rho_{\fl}$  have the same sign in von Mises sine and cosine models if and only if $E[\cos (\Theta-\mu_1) \cos (\Phi-\mu_2)] \geq 0$. Also, from Propositions~\ref{vms_expec_res}~\ref{E_cos_x_cos_y} and \ref{vmc_expec_res}~\ref{E_c_cos_x_cos_y}, we have $E[\cos (\Theta-\mu_1) \cos (\Phi-\mu_2)] = \frac{\partial^2 C_\alpha}{\partial \kappa_1 \partial \kappa_2}$ ($\alpha = s, c$)  for both sine and cosine models. Now, for the sine model, note that (see the infinite series representation (\ref{del_C_k1_k2_expr})) $\frac{\partial^2 C_s}{\partial \kappa_1 \partial \kappa_2} \geq 0$ for any $\kappa_1, \kappa_2, \kappa_3$ (and $\mu_1, \mu_2$). Thus, for the sine model the signs of   $\rho_{\js}$ and $\rho_{\fl}$ \emph{always agree} (although they may differ in magnitude). 
	 
	 In contrast, for the cosine model, $\frac{\partial^2 C_c}{\partial \kappa_1 \partial \kappa_2} \geq 0$ when $\kappa_3 \geq 0$ (see the infinite series representation (\ref{del_C_c_k1_k2_expr})), and hence the signs of $\rho_{\fl}$ and $\rho_{\js}$ are the same when $\kappa_3 \geq 0$. However, if $\kappa_3 < 0$ and $|\kappa_3|$ is large compared to $\kappa_1$ and $\kappa_2$,  $E[\cos (\Theta-\mu_1) \cos (\Phi-\mu_2)]$ can be negative, in which case $\rho_{\fl}$ and $\rho_{\js}$ will have opposite signs. In such cases, interpretations of the two correlation coefficients are not straightforward, especially when their magnitudes are high (see Section~\ref{sec_illus} for an example).  
	\end{remark}

	\begin{remark} \label{rem_sin_lambda_flip}
		Observe for the sine model that $|{\partial C_s}/{\partial \kappa_3}|$ (see (\ref{del_C_lambda_expr})), and hence $|\rho_{\fl}|$ and $|\rho_{\js}|$,  remain unchanged if the sign of $\kappa_3$ is flipped. This means, for fixed $\kappa_1$ and $\kappa_2$, reversing the sign of $\kappa_3$ just reverses the direction of association (as depicted by both $\rho_{\js}$ and $\rho_{\fl}$), while keeping the magnitude unchanged in the sine model. This can also be seen as a corollary to the fact that $(\Theta, \Phi) \sim \vms(\kappa_1, \kappa_2, \kappa_3, \mu_1, \mu_2)$ implies, and is implied by, $(\Theta, -\Phi) \sim \vms(\kappa_1, \kappa_2, -\kappa_3, \mu_1, \mu_2)$). However, this is not true in general for the cosine model. (See section~\ref{sec_illus} for examples).  
	\end{remark}

	As mentioned in Introduction, under certain conditions, both sine and cosine model densities closely approximate the normal density. In such cases, both $\rho_{\js}$ and $\rho_{\fl}$ are well approximated by the associated correlation parameter of the (approximate) normal distribution. The following two corollaries formally describe the situations where $\rho_{\fl}$ and $\rho_{\js}$ are approximately equal due to approximate normality of the sine and cosine model, and provide their common approximate values.

    \begin{corollary} \label{cor_circ_lin_reln_vmsin}
     Let $(\Theta, \Phi) \sim \vms(\kappa_1, \kappa_2, \kappa_3, \mu_1, \mu_2)$. If $\kappa_1$ and $\kappa_2$ are large and $\kappa_3^2 < \kappa_1 \kappa_2$, then $\rho_{\fl}(\Theta, \Phi) \approx \rho_{\js}(\Theta, \Phi) \approx \kappa_3/\sqrt{\kappa_1\kappa_2}$.   
    \end{corollary}
    
    \begin{proof}
    	Clearly, if $\kappa_3 = 0$, then $\Theta$ and $\Phi$ are independent, and hence $\rho_{\fl}(\Theta, \Phi) = \rho_{\js}(\Theta, \Phi) = 0 = \kappa_3/\sqrt{\kappa_1\kappa_2}$. So, without loss of generality, we assume $\kappa_3 \neq 0$.  From \citet[Proposition~2]{rivest:1988} and \citet[Section~2]{singh:2002} it follows that when $\kappa_1, \kappa_2$ are large and $\kappa_3^2 < \kappa_1 \kappa_2$, then $(\Theta, \Phi)$ have an approximately bivariate normal distribution with covariance matrix
    	\[
    	\Sigma = \frac{1}{\kappa_1\kappa_2 - \kappa_3^2}
    	\begin{pmatrix}
    	\kappa_2 & \kappa_3 \\
    	\kappa_3 & \kappa_1
    	\end{pmatrix}.
    	\]
    	Hence, in such cases, $\rho(\Theta, \Phi) \approx \frac{\kappa_3}{\sqrt{\kappa_1 \kappa_2}}$ from the dominated convergence theorem, where $\rho(X, Y)$ denotes the product moment correlation coefficient between $X$ and $Y$. Now, observe that when $\kappa_1 \kappa_2 > \kappa_3^2$, then both $A(\kappa_1)$ and $A(\kappa_2)$ are 
    	trivially bounded above by $\kappa_1 \kappa_2/\kappa_3^2$, where $A(x) = I_1(x)/I_0(x)$. This ensures unimodality of the marginal distributions of $\Theta$ and $\Phi$ \citep[Theorem~3]{singh:2002}. Furthermore, since $\kappa_1$ and $\kappa_2$ are large subject to $\kappa_1 \kappa_2 > \kappa_3^2$, the marginal distributions of $\Theta$ and $\Phi$ are highly concentrated (see, e.g., \cite[Proposition~2]{rivest:1988}). Therefore, it follows that $\rho_{\js}(\Theta, \Phi) \approx \rho(\Theta, \Phi)$ \citep[Theorem~2.1(f)]{jammalamadaka:1988} and $\rho_{\fl}(\Theta, \Phi) \approx \rho(\Theta, \Phi)$ \citep[property~(v) on p.~329]{fisher:1983}. This completes the proof. 
   \end{proof}

   \begin{corollary} \label{cor_circ_lin_reln_vmcos}
   	Let $(\Theta, \Phi) \sim \vmc(\kappa_1, \kappa_2, \kappa_3, \mu_1, \mu_2)$. If $\kappa_1$ and $\kappa_2$ are large and $\kappa_3 \geq - \kappa_1 \kappa_2 / (\kappa_1 + \kappa_2)$, then $\rho_{\fl}(\Theta, \Phi) \approx \rho_{\js}(\Theta, \Phi) \approx \kappa_3/\sqrt{(\kappa_1+\kappa_3)(\kappa_2+\kappa_3)}$.   
   \end{corollary}

  \begin{proof}
  	Without loss of generality let $\kappa_3 \neq 0$.  From \citep[Proposition~2]{rivest:1988} and \citep[Theorem~1]{mardia:2007}, it follows that when $\kappa_1, \kappa_2$ are large and $\kappa_3 \geq - \kappa_1 \kappa_2 / (\kappa_1 + \kappa_2)$, then $(\Theta, \Phi)$ is approximately bivariate normal with covariance matrix
  	\[
  	\Sigma = \frac{1}{\kappa_1\kappa_2 + (\kappa_1+\kappa_2)\kappa_3}
  	\begin{pmatrix}
  	\kappa_2 + \kappa_3 & \kappa_3 \\
  	\kappa_3 & \kappa_1 + \kappa_3
  	\end{pmatrix}.
  	\]  
  	Consequently, the marginal distributions of $\Theta$ and $\Phi$ are unimodal (approximately univariate normal), and are highly concentrated (since $\kappa_1, \kappa_2$ are large). The proof is completed by using arguments similar to the proof of	Corollary~\ref{cor_circ_lin_reln_vmsin}. 
  \end{proof}


	\subsection{Illustrations} \label{sec_illus}
	
	We now provide numerical and visual illustrations of  the circular variance and the two circular correlation coefficients,  which depict the spread and the toroidal linear association  between the coordinates of bivariate von Mises random deviates. For the sine model, we consider $\mu_1 = \mu_2 = 0$ and $\kappa_1 = \kappa_2 = \kappa$, so that $\Theta$ and $\Phi$ have the same marginal distributions. We choose three sets of values for $\kappa$, one moderate (1), one small (0.1) and one large (10). For each set, we consider four different values of $\kappa_3$, namely, $\kappa/2, -\kappa/2, 2\kappa$ and $-2\kappa$. For each of these 12 combinations, we compute $\rho^s =  {\kappa_3}/{\sqrt{\kappa_1\kappa_2}}$,  $\rho_{\fl}$, $\rho_{\js}$ and $\var(\Theta) =\var(\Phi)$, using formulas provided in Theorem~\ref{thm_vms}. To note the accuracies of the formulas, we also compute Monte Carlo estimates $\hat\rho_{\fl}$, $\hat\rho_{\js}$ and $\hat \var(\Theta)$ along with their estimated standard errors,  on the basis of 100 replicated random samples of size 10,000 each, generated from a von Mises sine population for each respective combination of parameters, and compare the estimates with their true analytical counterparts. Analogous computations  are performed for the cosine model, with $\kappa_3$ replaced by $\kappa_3 = \kappa/2, -\kappa/2, 2\kappa, -2\kappa$, and $\rho^s$ replaced by $\rho^c=  \kappa_3/ \sqrt{(\kappa_1+\kappa_3) (\kappa_2+\kappa_3)}$.  The resulting values are shown in Tables~\ref{tab_vms} and \ref{tab_vmc}. All computations are done in R using the package \texttt{BAMBI} \citep{bambi_rpack}, in which we have incorporated functions for calculating circular variance and correlations, both theoretical (obtained from the analytical formulas), and estimated (from sample data matrices), along with functions for random simulation from these bivariate angular distributions.

	\begin{table}[ht]
	
		\tiny
		\centering
		\begin{tabular}{|rrrrrrrrrr|}
			
			\midrule
			$\kappa_1$ & $\kappa_2$ & $\kappa_3$ & $\rho^s$ & $\rho_\js$ & $\hat\rho_\js$ & $\rho_\fl$ & $\hat\rho_\fl$ & $\var(\Theta)$ & $\hat\var(\Theta)$\\ 
			\midrule	
		1 &   1 & 0.5 & 0.5 & 0.22 & 0.22 (0.0089) & 0.078 & 0.079 (0.0038) & 0.56 & 0.56 (0.0066) \\ 
		1 &   1 & -0.5 & -0.5 & -0.22 & -0.22 (0.0089) & -0.078 & -0.078 (0.0038) & 0.56 & 0.56 (0.0060) \\ 
		1 &   1 &   2 &   2 & 0.70 & 0.70 (0.0049) & 0.23 & 0.23 (0.0077) & 0.62 & 0.62 (0.0064) \\ 
		1 &   1 &  -2 &  -2 & -0.70 & -0.70 (0.0049) & -0.23 & -0.23 (0.0069) & 0.62 & 0.63 (0.0060) \\ 
		0.1 & 0.1 & 0.05 & 0.5 & 0.025 & 0.024 (0.010) & 0.00012 & 0.00010 (0.00026) & 0.95 & 0.95 (0.0070) \\ 
		0.1 & 0.1 & -0.05 & -0.5 & -0.025 & -0.025 (0.010) & -0.00012 & -0.000092 (0.00030) & 0.95 & 0.95 (0.0072) \\ 
		0.1 & 0.1 & 0.2 &   2 & 0.10  & 0.097 (0.0094) & 0.00054 & 0.00039 (0.0010) & 0.95 & 0.95 (0.0069) \\ 
		0.1 & 0.1 & -0.2 &  -2 & -0.10 & -0.098 (0.011) & -0.00054 & -0.00041 (0.0010) & 0.95 & 0.95 (0.0071) \\ 
		10 &  10 &   5 & 0.5 & 0.46 & 0.46 (0.0080) & 0.46 & 0.46 (0.0079) & 0.064 & 0.064 (0.00088) \\ 
		10 &  10 &  -5 & -0.5 & -0.46 & -0.46 (0.0073) & -0.46 & -0.45 (0.0073) & 0.064 & 0.064 (0.00097) \\ 
		10 &  10 &  20 &   2 & 0.98 & 0.98 (0.00030) & 0.89 & 0.89 (0.0017) & 0.49 & 0.49 (0.0020) \\ 
		10 &  10 & -20 &  -2 & -0.98 & -0.98 (0.00030) & -0.89 & -0.89 (0.0017) & 0.49 & 0.49 (0.0021) \\   
			\bottomrule
		\end{tabular}
		\caption{The true (analytical) correlations $\rho_{\js}$ and $\rho_{\fl}$ and variance $\var(\Theta)$ along with their sample estimates for the \textbf{von Mises sine model} for various choices of $\kappa_1, \kappa_2$ and $\kappa_3$. The numbers within the parentheses denote the standard errors of the associated Monte carlo estimates, and $\rho^s =  {\kappa_3}/{\sqrt{\kappa_1\kappa_2}}$.}
	
	   \label{tab_vms}
	\end{table}
	
	\begin{table}[ht]

		\centering
		\tiny
		\begin{tabular}{|rrrrrrrrrr|}
			
			\midrule
			$\kappa_1$ & $\kappa_2$ & $\kappa_3$ & $\rho^c$ &  $\rho_\js$ & $\hat\rho_\js$ & $\rho_\fl$ & $\hat\rho_\fl$ & $\var(\Theta)$ & $\hat\var(\Theta)$\\ 
			\midrule
			  1 &   1 & 0.5 & 0.33 & 0.21 & 0.21 (0.0098) & 0.12 & 0.12 (0.0056) & 0.48 & 0.48 (0.0057) \\ 
			1 &   1 & -0.5 &  -1 & -0.22 & -0.22 (0.010) & -0.025 & -0.025 (0.0026) & 0.64 & 0.64 (0.0061) \\ 
			1 &   1 &   2 & 0.67 & 0.61 & 0.61 (0.0062) & 0.52 & 0.52 (0.0057) & 0.37 & 0.37 (0.0050) \\ 
			1 &   1 &  -2 &  -2 & -0.68 & -0.68 (0.0071) & 0.37 & 0.37 (0.0062) & 0.84 & 0.84 (0.0065) \\ 
			0.1 & 0.1 & 0.05 & 0.33 & 0.025 & 0.024 (0.011) & 0.00075 & 0.00068 (0.00036) & 0.95 & 0.95 (0.0065) \\ 
			0.1 & 0.1 & -0.05 &  -1 & -0.025 & -0.025 (0.011) & 0.00049 & 0.00054 (0.00039) & 0.95 & 0.95 (0.0072) \\ 
			0.1 & 0.1 & 0.2 & 0.67 & 0.099 & 0.098 (0.010) & 0.010 & 0.010 (0.0013) & 0.95 & 0.95 (0.0068) \\ 
			0.1 & 0.1 & -0.2 &  -2 & -0.099 & -0.097 (0.012) & 0.0094 & 0.0095 (0.0015) & 0.95 & 0.95 (0.0070) \\ 
			10 &  10 &   5 & 0.33 & 0.33 & 0.33 (0.0083) & 0.33 & 0.33 (0.0083) & 0.038 & 0.038 (0.00050) \\ 
			10 &  10 &  -5 &  -1 & -0.65 & -0.64 (0.0051) & -0.62 & -0.62 (0.0050) & 0.15 & 0.15 (0.0019) \\ 
			10 &  10 &  20 & 0.67 & 0.67 & 0.67 (0.0051) & 0.67 & 0.67 (0.0051) & 0.030 & 0.030 (0.00044) \\ 
			10 &  10 & -20 &  -2 & -0.97 & -0.97 (0.0017) & 0.61 & 0.60 (0.0061) & 0.81 & 0.81 (0.0053) \\    
			\bottomrule
		\end{tabular}
		\caption{The true (analytical) correlations $\rho_{\js}$ and $\rho_{\fl}$ and variance $\var(\Theta)$ along with their sample estimates for the \textbf{von Mises cosine model} for various choices of $\kappa_1, \kappa_2$ and $\kappa_3$. The numbers within the parentheses denote the standard errors of the associated Monte carlo estimates, and  $\rho^c = \frac{\kappa_3}{\sqrt{(\kappa_1+\kappa_3)(\kappa_2+\kappa_3)}}$.}
		
		\label{tab_vmc}
	\end{table}
	
    The noticeable similarities between the true and the estimated values depicted in Tables~\ref{tab_vms} and \ref{tab_vmc} (together with the small standard errors) demonstrate the accuracies of the formulas. As expected, reversing the sign of $\kappa_3$ while keeping $\kappa_1$ and $\kappa_2$ unchanged has no impact on $\var(\Theta)$, and only reverses the signs of $\rho_{\js}$ and $\rho_{\fl}$ in the sine model (see Remark~\ref{rem_sin_lambda_flip}). This however does not generally hold for the cosine model. For both sine and cosine models, larger $\kappa_1$ and $\kappa_2$ values induce higher concentrations  when $\kappa_3$ is moderate, as reflected by the smaller variances.
    
    For both sine and cosine models, the numerical results show that for fixed $\kappa_1$, $\kappa_2$, increasing the ``covariance'' parameter in absolute value increases the magnitude of the (T-linear) association, as reflected in the $\rho_{\js}$ values. As expected, in each case $\sgn(\rho_{\js})$ is the same as  $\sgn(\kappa_3)$ (Corollary~\ref{cor_js_sign}) and $|\rho_{\fl}| \leq |\rho_{\js}|$ (Corollary~\ref{cor_fl_js_reln}). Note that $\rho_{\fl}$ and $\rho_{\js}$ are both close and well approximated by $\rho^s$ for the sine model in the case $\kappa_1 = \kappa_2 = 10$ (large) and $\kappa_3 = 5, -5$ (so that $\kappa_3^2 < \kappa_1 \kappa_2$), consistent with the result of Corollary~\ref{cor_circ_lin_reln_vmsin}.  A similar observation holds for the cosine model in the case $\kappa_1 = \kappa_2 = 10$ and $\kappa_3  = 5, 10$ (so that $\kappa_3 \geq - \kappa_1 \kappa_2 / (\kappa_1 + \kappa_2)$), see Corollary~\ref{cor_circ_lin_reln_vmcos}. It is interesting to note that the signs of the two circular correlations differ in the cosine model when $\kappa_3$ is very negative compared to $\kappa_1$ and $\kappa_2$, which corresponds a bimodal density (see \cite{mardia:2007}). For example, we see that when $\kappa_1 = \kappa_2 = 10$ and $\kappa_3 = -20$, we get $\rho_{\js} = -0.97$, while $\rho_{\fl} = 0.61$.

	\begin{figure}[!hptb]
		\centering
	\subfloat[$\kappa_3 = 0.5, \rho_\js = 0.22,  \rho_\fl = 0.078$]%
	{\includegraphics[width = 0.45\linewidth, height = 0.195\textheight,  page=1]{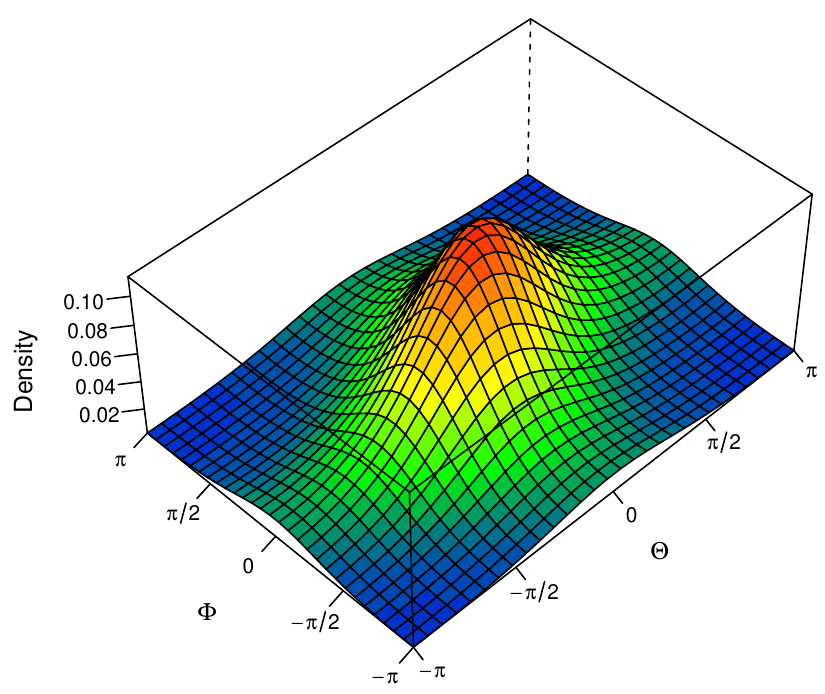}} 
	\hfill
	\subfloat[$\kappa_3 = 0.5,	\rho_\js = 0.21,  \rho_\fl = 0.12$]%
	{\includegraphics[width = 0.485\linewidth, height = 0.19\textheight,  page=1]{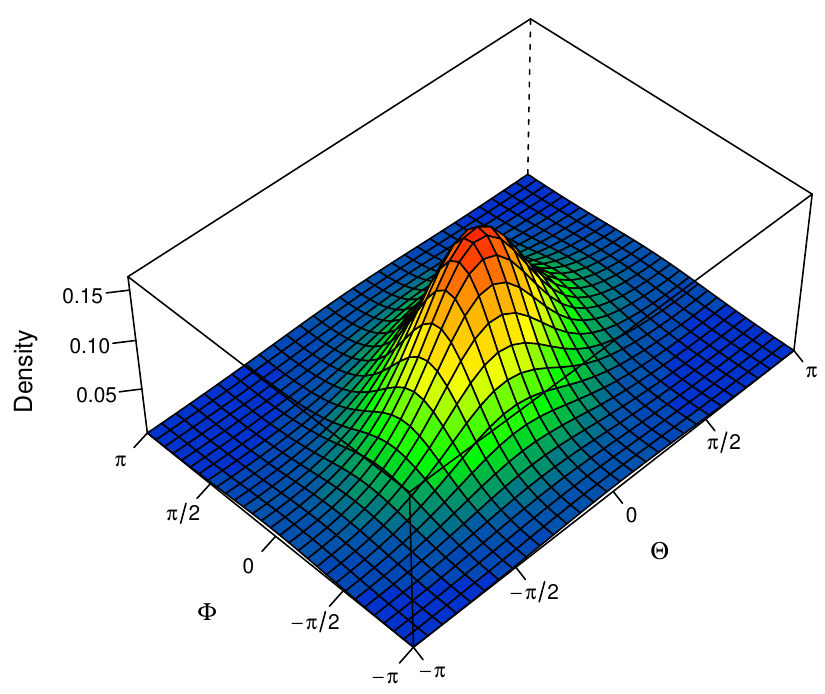}} \\
	\subfloat[$\kappa_3 = -0.5, \rho_\js = -0.22,  \rho_\fl = -0.078$]%
	{\includegraphics[width = 0.485\linewidth, height = 0.19\textheight,  page=2]{plots/vmsin_plots_cropped.pdf}} 
	\hfill
	\subfloat[$\kappa_3 = -0.5,	\rho_\js = -0.22,  \rho_\fl = -0.025$]%
	{\includegraphics[width = 0.485\linewidth, height = 0.19\textheight,  page=2]{plots/vmcos_plots_cropped.pdf}} \\
	\subfloat[$\kappa_3 = 2,	\rho_\js = 0.70,  \rho_\fl = 0.23$]
	{\includegraphics[width = 0.485\linewidth, height = 0.19\textheight,  page=3]{plots/vmsin_plots_cropped.pdf}} 
	\hfill
	\subfloat[$\kappa_3 = 2,	\rho_\js = 0.61,  \rho_\fl = 0.52$]%
	{\includegraphics[width = 0.485\linewidth, height = 0.19\textheight,  page=3]{plots/vmcos_plots_cropped.pdf}} \\
	\subfloat[$\kappa_3 = -2,	\rho_\js = -0.70,  \rho_\fl = -0.23$]%
	{\includegraphics[width = 0.485\linewidth, height = 0.19\textheight,  page=4]{plots/vmsin_plots_cropped.pdf}} 
	\hfill
	\subfloat[$\kappa_3 = -2,	\rho_\js = -0.68,  \rho_\fl = 0.37$]%
	{\includegraphics[width = 0.485\linewidth, height = 0.19\textheight,  page=4]{plots/vmcos_plots_cropped.pdf}}
	\caption{Density surfaces provide a visual assessment of circular correlations for von Mises sine (left) and cosine (right) models with parameters $\kappa_1 = \kappa_2 = 1$ and $\mu_1 = \mu_2 = 0$. There are four plots for each model, for each of $\kappa_3$ or $\kappa_3 \in \left\{0.5, -0.5, 2, 2\right\}$.} 
	\label{fig_vm_den_plots}	
	\end{figure}
	
	To visualize the two circular correlations and  how they describe the T-linear associations in the two models, we plot the density surfaces corresponding to the four parameter combinations $\kappa_1 = \kappa_2 = 1$ and $\kappa_3 = -0.5, 0.5, -2, 2$ (Figure~\ref{fig_vm_den_plots}). When $|\kappa_3|$ is small (0.5 or -0.5, Figure~\ref{fig_vm_den_plots}(a),(c)) in the sine model, or $\kappa_3$ is positive (2 or -2, Figure~\ref{fig_vm_den_plots}(b),(f)) in the cosine model, the respective densities are unimodal, and the direction of association matches the sign of the ``covariance'' parameter.  For the sine model, when $|\kappa_3|$ is large compared to $\kappa_1, \kappa_2$ ($\kappa_3 =$ 2 and -2, Figure~\ref{fig_vm_den_plots}(e),(g)), bimodality is induced  \citep{singh:2002, mardia:2007}.  Bimodality in cosine requires very negative $\kappa_3$ values compared to $\kappa_1$, $\kappa_2$ as seen in Figure~\ref{fig_vm_den_plots}(h).	Next, we see that reversing the sign of $\kappa_3$ for fixed $\kappa_1$ and $\kappa_2$ in the sine model simply reverses the direction of the association between $\Theta$ and $\Phi$, and the signs of $\rho_{\fl}$ and $\rho_{\js}$ always agree. The signs of $\rho_{\fl}$ and $\rho_{\js}$ also agree in the cosine model when $\kappa_3 > 0$. However, if $\kappa_3 < 0$, the association can be difficult to interpret, especially when the density is bimodal. For example, in Figure~\ref{fig_vm_den_plots}(d) where $\kappa_3 = -0.5$, a negative association is visible and both $\rho_{\fl}$ and $\rho_{\js}$ are negative. In contrast, Figure~\ref{fig_vm_den_plots}(h) shows an example of a bimodal density where $\rho_{\js}$ is  negative but $\rho_{\fl}$ is  positive -- visually, $\Theta$ and $\Phi$ are positively associated locally around the two peaks; however, their overall association is hard to interpret.

	\section{Inference for circular correlation coefficients from sample data}\label{sec:inference}
   	For a key practical use of the proposed circular correlation coefficient formulas, we consider the modeling of bivariate angular data using the von Mises sine or cosine distributions.  In this context of parametric modeling, we are interested in making formal statistical inferences on the JS and FL circular correlation coefficients between the random coordinates of the angle pair. We focus on likelihood-based frequentist inference using asymptotic normal distributions; corresponding Bayesian inference can be straightforwardly made using the presented formulas applied on posterior samples (e.g., obtained via MCMC \citep{bambi_rpack} or variational Bayes).  Note that such model-based inference is useful when a parametric model can be reasonably assumed for a given dataset; otherwise, the circular correlation coefficients can be estimated using their non-parametric formulas and associated theory for inference  \citep{fisher:1983, jammalamadaka:1988}. 
   	
   	
   	Let $\{(\Theta_i, \Phi_i): i = 1, \dots, n\}$ be random samples from a bivariate von Mises distribution (either sine or cosine variant) $f(\cdot \mid \etab)$ with parameter vector $\etab = (\mu_1, \mu_2, \kappa_1, \kappa_2, \kappa_3)$ belonging to the associated parameter space:
   	\[
   	\Omega = \{0 \leq \kappa_1 < \infty, 0 \leq \kappa_2 < \infty, -\infty < \kappa_3 < \infty, -\pi \leq \mu_1 \leq \pi, -\pi \leq \mu_2 \leq \pi \}.
   	\]
   	The maximum likelihood estimator $\hat\etab$ of $\etab$ the parameters is obtained by maximizing the log likelihood $l_n(\etab) := \sum_{i=1}^{n} \log f(\Theta_i, \Phi_i \mid \etab)$. Let $l_n'(\etab)$, $l_n''(\etab)$ and $l_n'''(\etab)$ denote the first, second and third derivatives of $l_n$ with respect to $\etab$ respectively (they all exist), $\etab_0$ be the true value of $\etab$, and $B(\etab_0)$ denote a neighborhood around $\etab_0$. It is straightforward to show that the following regularity conditions of Self and Liang \citep{self:liang:1987} holds \citep[see also][]{shieh:johnson:2005}:
   	\begin{enumerate}[label = (\roman*)]
   		\item The first three derivatives of $l(\etab)$ with respect to $\etab$ exists almost surely  on the intersection of $B(\etab_0)$ and $\Omega$,
   		
   		\item There exists a function $M((\Theta_1, \Phi_1), ..., (\Theta_n, \Phi_n))$ that bounds the absolute value of each entry of $l'''(\etab)$ and has $E_{\etab} (M) < \infty$ for all $\etab$ in the intersection of $B(\etab_0)$ and $\Omega$,  and
   		
   		\item The Fisher information matrix $I(\etab) := - E[\frac{\partial^2}{\partial \etab \partial \etab^T} \log f(\Theta_1, \Phi_1 \mid \etab)]$ is positive definite for all $\etab$ on $B(\etab_0)$ and $I(\etab_0)$ is equal to the covariance matrix of $n^{-1/2} l'_n (\etab_0)$.
   	\end{enumerate}
    
    \noindent Consequently, the maximum likelihood estimator $\hat \etab$ attains the asymptotic normal distribution:
    \[
    \sqrt{n} \left(\hat \etab -\etab_0\right) \xrightarrow{d} N_5(0, I(\etab_0)^{-1}) \text{ as } n \to \infty.
    \]
    The asymptotic distribution of the parametric circular coefficient estimate  $\hat\rho = \rho(\hat \etab)$ (either Fisher-Lee or Jamalamadaka-Sarma form) is then obtained using the delta method as:
    \[
    \sqrt{n} \left(\hat \rho - \rho_0\right) \xrightarrow{d} N\left(0, \nabla\rho(\etab_0)^T I(\etab_0)^{-1} \nabla\rho(\etab_0) \right)
    \]
    where $\rho_0 = \rho(\etab_0)$ is the true value of $\rho$ and $\nabla\rho(\etab_0) = \frac{\partial}{\partial \etab} \rho(\etab) |_{\etab_0 = \etab_0}$. (Note that the elements of $\nabla\rho(\etab_0)$ corresponding to $\mu_1$ and $\mu_2$ are both zero.) The sample analogue $\hat I := -n^{-1} \sum_{i=1}^n \left. l''(\etab)\right|_{\etab = \hat\etab}$ of the population Fisher information matrix $I(\etab)$  together with the above asymptotic distribution thus yields via Slutsky's theorem the following approximate normal distribution of the maximum likelihood estimate $\hat\rho$:
    \[
    \hat \rho \stackrel{a}{\sim} N \left(\rho_0, \frac{V}{n} \right), \text{ where } V = \nabla\rho(\hat \etab)^T \hat I^{-1} \nabla\rho(\hat \etab)
    \] 
    for sufficiently large $n$. This approximate distribution serves as a basis for making inference on the presented parametric circular correlation coefficients.
    
    To aid comparison, we also consider inference for the circular correlation coefficients based on their non-parametric sample estimates and the associated asymptotic normal distributions described in \citep{fisher:1983, jammalamadaka:1988}. It is to be noted that  the non-parametric estimate of the Fisher-Lee circular correlation coefficient is asymptotically biased when the population correlation is non-zero; the authors suggest using the jackknife for de-biasing and estimation of asymptotic standard error in general. We elected to use the simplified approximate standard error formula provided in \cite{fisher:1983} that is asymptotically correct when the population correlation coefficient is zero, for fair comparison with the maximum likelihood and non-parametric Jammalamadaka-Sarma estimators which do not require computationally-intensive methods for approximate estimation of standard error.

    \subsection{Simulation Study}
     We conducted simulation experiments to understand the effectiveness of  formula-based parametric estimates and the non-parametric  estimates of the Fisher-Lee and the Jammalamadaka-Sarma correlation coefficients in practice. We generated 1000 replicated datasets with sample sizes $n \in \{50, 100, 500, 1000\}$ separately from the von Mises sine and cosine distributions with parameters $\kappa_1 = \kappa_2 = 1$, $\mu_1 = \mu_2 = 0$ and $\kappa_3 \in \{-2, -0.5, 0, 0.5, 2\}$. In each generated dataset we obtained the parametric maximum likelihood and non-parametric estimates of $\rho_{\fl}$ and $\rho_{\js}$, computed the associated approximate standard errors obtained from the asymptotic normal distribution, and finally obtained an approximate 95\% confidence interval of the form 
     $$
     \mbox{estimate} \pm 1.96 \times \mbox{standard error}
     $$
     based on each estimated correlation coefficient. For each model (von Mises sine or cosine), $\kappa_3$, and sample size $n$, this yielded 1000 replicates of  approximate 95\% confidence intervals for $\rho_{\fl}$ and $\rho_{\js}$. The \textit{frequentist coverages} of these confidence intervals, i.e.,  the proportions of replicates where the computed confidence intervals contain the actual population correlation coefficients,  were subsequently obtained. Note that the purpose of this simulation is to demonstrate the usefulness of the presented formulas in applications where the parametric modeling is reasonable; as such, the behaviors of the maximum likelihood estimators under model misspecification is not considered here. 
    
     Tables~\ref{tab:vmsin_coverage} and \ref{tab:vmcos_coverage} display the frequentist coverages of the approximate 95\% confidence intervals for von Mises sine and cosine model respectively. For both models, coverages of the confidence intervals associated with the maximum likelihood estimators are moderately adequate for $n = 50$, and are adequate for $n \geq 100$. The non-parametric estimates for $\rho_{\js}$  show adequate coverages for $n \geq 100$ in the sine model, and in the cosine model with $\kappa_3 \neq -2$; when $\kappa_3 = -2$ the coverage is poor in cosine model even when the sample size is as large as 1000. This is not particularly surprising as the cosine distribution is very bimodal when $\kappa_3 = -2$ with two positively associated local clusters around the ``off diagonal'' line (see Figure~\ref{fig_vm_den_plots}(h) and the corresponding discussion in Section~\ref{sec_illus});  a  much larger sample size is thus seemingly required for an accurate representation of the population, and hence an accurate estimation of $\rho_{\js}$ in this setting. Note that the confidence intervals based on the  maximum likelihood estimate has adequate coverages, demonstrating the utility of parametric knowledge in such cases. 
    
	\begin{table}[htbp]
		\centering
		\begin{tabular}{|c|c|r|rrrrr|}
			\toprule
			Param. & Estimate & \multicolumn{1}{c|}{$\kappa_3$} & \multicolumn{1}{c}{$n = 50$} & \multicolumn{1}{c}{$n = 100$} & \multicolumn{1}{c}{$n = 500$} & \multicolumn{1}{c}{$n = 1000$} & \multicolumn{1}{c|}{$n = 5000$} \\
			\midrule
			\multirow{10}[4]{*}{$\rho_{\fl}$} & \multirow{5}[2]{*}{MLE} & -2    & 0.936 & 0.945 & 0.953 & 0.942 & 0.948 \\
			&       & -0.5  & 0.933 & 0.950 & 0.950 & 0.950 & 0.947 \\
			&       & 0     & 0.949 & 0.939 & 0.955 & 0.954 & 0.945 \\
			&       & 0.5   & 0.929 & 0.946 & 0.960 & 0.949 & 0.953 \\
			&       & 2     & 0.946 & 0.947 & 0.956 & 0.963 & 0.939 \\
			\cmidrule{2-8}          & \multirow{5}[2]{*}{Non-param.} & -2    & 0.593 & 0.616 & 0.618 & 0.624 & 0.632 \\
			&       & -0.5  & 0.873 & 0.907 & 0.932 & 0.924 & 0.911 \\
			&       & 0     & 0.909 & 0.935 & 0.948 & 0.951 & 0.944 \\
			&       & 0.5   & 0.883 & 0.888 & 0.918 & 0.918 & 0.932 \\
			&       & 2     & 0.652 & 0.652 & 0.629 & 0.648 & 0.617 \\
			\midrule
			\multirow{10}[4]{*}{$\rho_{\js}$} & \multirow{5}[2]{*}{MLE} & -2    & 0.932 & 0.938 & 0.942 & 0.945 & 0.942 \\
			&       & -0.5  & 0.911 & 0.938 & 0.950 & 0.947 & 0.949 \\
			&       & 0     & 0.906 & 0.930 & 0.948 & 0.953 & 0.945 \\
			&       & 0.5   & 0.911 & 0.940 & 0.955 & 0.947 & 0.961 \\
			&       & 2     & 0.912 & 0.935 & 0.959 & 0.950 & 0.954 \\
			\cmidrule{2-8}          & \multirow{5}[2]{*}{Non-param.} & -2    & 0.913 & 0.934 & 0.938 & 0.934 & 0.945 \\
			&       & -0.5  & 0.924 & 0.933 & 0.950 & 0.948 & 0.952 \\
			&       & 0     & 0.918 & 0.931 & 0.952 & 0.955 & 0.944 \\
			&       & 0.5   & 0.923 & 0.939 & 0.960 & 0.947 & 0.960 \\
			&       & 2     & 0.880 & 0.932 & 0.946 & 0.948 & 0.951 \\
			\bottomrule
		\end{tabular}%
		\caption{Frequentist coverages of approximate 95\% confidence intervals based on maximum likelihood and non-parametric estimates of Fisher-Lee and Jammalamadaka-Sarma correlation coefficients for von Mises sine model with $\kappa_1 = \kappa_2 = 1$, $\mu_1 = \mu_2 = 0$ and various $\kappa_3$. Sample sizes $\{n\}$ are displayed along the columns, and the numbers inside the cells display the observed coverages of the associated approximate confidence intervals obtained from replicated simulated datasets.} 
				\label{tab:vmsin_coverage}
	\end{table}%

	\begin{table}[htbp]
		\centering
		\begin{tabular}{|c|c|r|rrrrr|}
			\toprule
			Param. & Estimate & \multicolumn{1}{c|}{$\kappa_3$} & \multicolumn{1}{c}{$n = 50$} & \multicolumn{1}{c}{$n = 100$} & \multicolumn{1}{c}{$n = 500$} & \multicolumn{1}{c}{$n = 1000$} & \multicolumn{1}{c|}{$n = 5000$} \\
			\midrule
			\multirow{10}[4]{*}{$\rho_{\fl}$} & \multirow{5}[2]{*}{MLE} & -2    & 0.922 & 0.941 & 0.952 & 0.946 & 0.949 \\
			&       & -0.5  & 0.996 & 0.991 & 0.957 & 0.959 & 0.937 \\
			&       & 0     & 0.958 & 0.931 & 0.946 & 0.950  & 0.940 \\
			&       & 0.5   & 0.925 & 0.938 & 0.937 & 0.947 & 0.957 \\
			&       & 2     & 0.921 & 0.940  & 0.949 & 0.968 & 0.959 \\
			\cmidrule{2-8}          & \multirow{5}[2]{*}{Non-param.} & -2    & 0.147 & 0.146 & 0.121 & 0.118 & 0.125 \\
			&       & -0.5  & 0.883 & 0.902 & 0.942 & 0.940  & 0.946 \\
			&       & 0     & 0.909 & 0.935 & 0.948 & 0.951 & 0.944 \\
			&       & 0.5   & 0.878 & 0.881 & 0.884 & 0.889 & 0.91 \\
			&       & 2     & 0.963 & 0.947 & 0.964 & 0.979 & 0.962 \\
			\midrule
			\multirow{10}[4]{*}{$\rho_{\js}$} & \multirow{5}[2]{*}{MLE} & -2    & 0.893 & 0.939 & 0.957 & 0.955 & 0.955 \\
			&       & -0.5  & 0.914 & 0.944 & 0.942 & 0.952 & 0.956 \\
			&       & 0     & 0.919 & 0.921 & 0.943 & 0.948 & 0.939 \\
			&       & 0.5   & 0.934 & 0.939 & 0.937 & 0.944 & 0.951 \\
			&       & 2     & 0.903 & 0.937 & 0.958 & 0.962 & 0.954 \\
			\cmidrule{2-8}          & \multirow{5}[2]{*}{Non-param.} & -2    & 0.394 & 0.500   & 0.705 & 0.760  & 0.880 \\
			&       & -0.5  & 0.929 & 0.948 & 0.932 & 0.945 & 0.943 \\
			&       & 0     & 0.918 & 0.931 & 0.952 & 0.955 & 0.944 \\
			&       & 0.5   & 0.924 & 0.947 & 0.949 & 0.949 & 0.961 \\
			&       & 2     & 0.896 & 0.906 & 0.948 & 0.950  & 0.941 \\
			\bottomrule
		\end{tabular}%
		\caption{Frequentist coverages of approximate 95\% confidence intervals based on maximum likelihood and non-parametric estimates of Fisher-Lee and Jammalamadaka-Sarma correlation coefficients for von Mises cosine model with $\kappa_1 = \kappa_2 = 1$, $\mu_1 = \mu_2 = 0$ and various $\kappa_3$. Sample sizes $\{n\}$ are displayed along the columns, and the numbers inside the cells display the observed coverages of the associated approximate confidence intervals obtained from replicated simulated datasets.} 
				\label{tab:vmcos_coverage}%
	\end{table}%

     It is to be noted that confidence intervals based on the non-parametric estimate of $\rho_{\fl}$ in both models have largely inadequate coverages for all sample sizes, unless $\kappa_3 = 0$. This is unsurprising as the estimate and the form of the approximate standard error used in our study are asymptotically biased for $\kappa_3 \neq 0$ as noted in \citep{fisher:1983}, and therefore requires bias adjustment prior to construction of confidence intervals in general.

\subsection{Real data example}
	
    We now apply the inference procedures to a sample of real bivariate angular data from a cell biophysics experiment.  In \citet{lan2018integrating}, the authors studied the cell migration characteristics of NIH 3T3 fibroblasts.  In doing so, individual cells were prepared and monitored via fluorescence microscopy on 2-D substrate over a period of 1 hour.  For each cell, images were captured at one minute intervals, from which quantitative cell features were extracted.  Here, we consider two of those features:  (i) the centroid of the cell, (ii) the centroid of the associated nucleus within the cell.  Both of these may be represented as 2-D coordinate pairs.  To gain insight into cell migration patterns, \citet{lan2018integrating} calculated the movement angles from these cell and nucleus coordinates.  Specifically, for an individual cell we have cell centroid coordinates $(x^c_t, y^c_t)$ and nucleus centroid coordinates $(x^n_t, y^n_t)$ recorded at $t= 0, 1, \ldots, 60$ minutes.  Then the angles
	$$
    \theta^c_t = \arctan \left[\frac{y^c_{t} - y^c_{t-1}}{x^c_{t} - x^c_{t-1}}\right],  \mbox{~~~~}   \theta^n_t = \arctan \left[\frac{y^n_{t} - y^n_{t-1}}{x^n_{t} - x^n_{t-1}}\right]
    $$
	were calculated for $t=1, \ldots, 60$.    

	To provide examples, the resulting scatterplots of the pairs $(\theta^c_t, \theta^n_t)$ for two individual cells are plotted in Figure \ref{fig_realdata_scatter}.  From visual examination, example cell A has an apparent slight positive circular correlation, while example cell B has a stronger positive circular correlation, recalling that the edges of the plot wrap around.  We focus on these two datasets for further illustration.  First, we fit the von Mises sine and cosine models to each dataset using maximum likelihood.  Then, we compute the estimates of the JS and FL correlation coefficients along with approximate 95\% confidence intervals, using the fitted parametric models and the corresponding non-parametric estimates.  Finally, we use the log-likelihood values to determine which parametric model is preferred, and simulate from the fitted parametric model to informally corroborate  goodness-of-fit of the assumed model.

	\begin{figure}[!hptb]
	\centering
	\subfloat[Example cell A]%
	{\includegraphics[width = 0.5\linewidth]{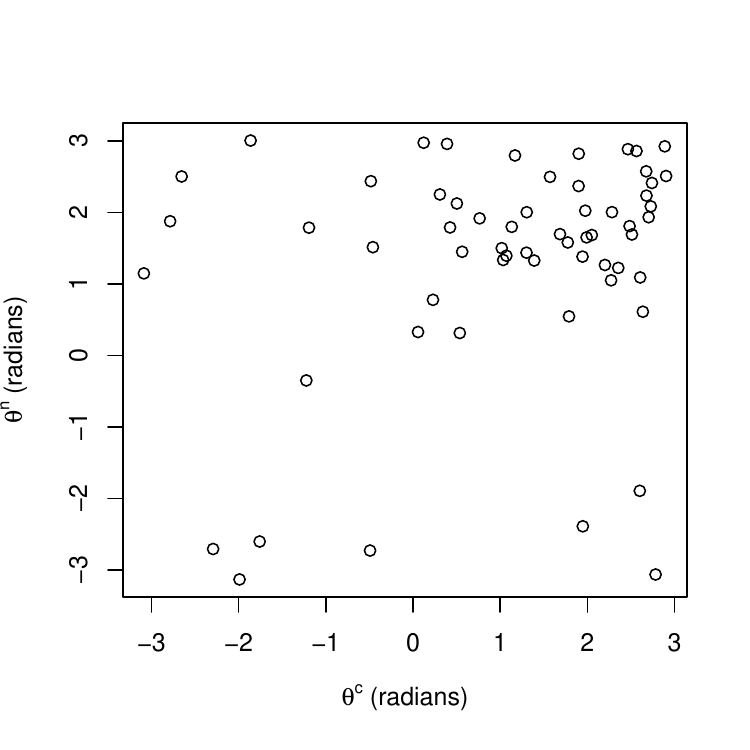}} 
	\hfill
	\subfloat[Example cell B]%
	{\includegraphics[width = 0.5\linewidth]{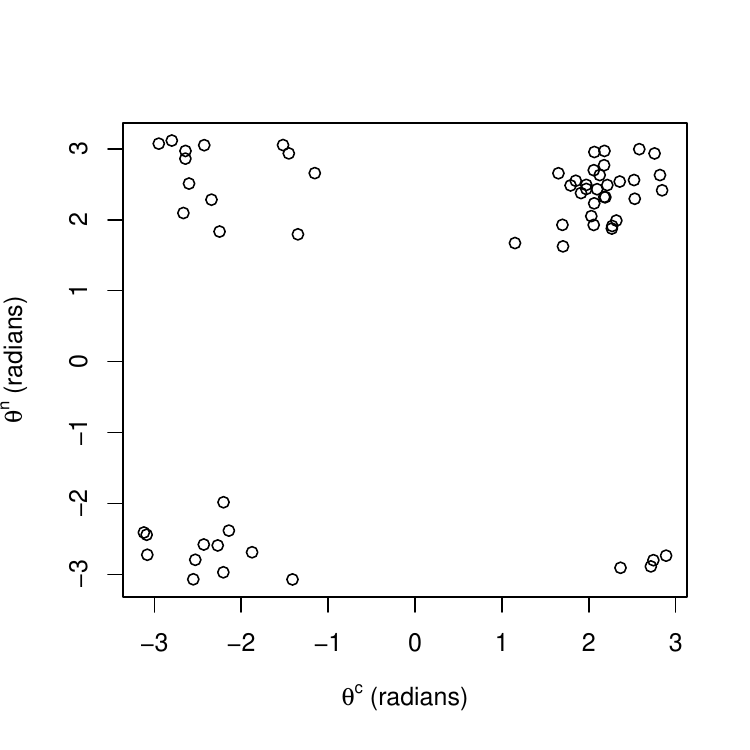}}
	\caption{Scatterplots of cell and nucleus movement angle pairs for two individual NIH 3T3 cells.} 
	\label{fig_realdata_scatter}	
\end{figure}

	The estimates of the circular correlation coefficients and associated 95\% confidence intervals are provided in Table~\ref{tab:realdata_ests} for these two datasets.  We may observe that the circular correlations are indeed stronger and more positive for cell B than cell A, and that the estimates in each case satisfy $\hat{\rho}_{\js} \ge \hat{\rho}_{\fl}$ (Corollary \ref{cor_fl_js_reln}).  For ${\rho}_{\js}$, the three point estimates roughly agree for both datasets, while it can be noted that the MLE-based confidence intervals are narrower than the non-parametric ones.  The same is true for ${\rho}_{\fl}$ in example cell B; for ${\rho}_{\fl}$ in example cell A, the non-parametric CI is slightly narrower, though may be subject to some undercoverage as suggested by the simulation studies in Tables~\ref{tab:vmsin_coverage} and \ref{tab:vmcos_coverage}.  The values of the log-likelihood at the MLEs (shown in the `log-lik' column of Table~\ref{tab:realdata_ests}) suggest that the cosine model may be a somewhat better fit than the sine model for cell A, while the sine model may be a better fit than the cosine model for cell B.  Since the sine and cosine models each have five parameters, model comparisons via AIC or BIC yield the same results:  the cosine model is preferred for cell A by a difference of 4.3 in AIC (or BIC), and the sine model is preferred for cell B by a difference of 3.3 in AIC (or BIC).  To corroborate the goodness-of-fit for these parametric models, a dataset of the same size (i.e., 60 angle pairs) is simulated from the best-fit model for each dataset and plotted in Figure \ref{fig_fitted_scatter}; the simulated datasets indeed largely resemble the real data.
	
    	\begin{table}[htbp]
    	\centering
    	\begin{tabular}{|c|c|c|cc|cc|}
    		\toprule
    		Dataset & Model & \multicolumn{1}{c|}{log-lik} & $\hat{\rho}_{\fl}$ & 95\% CI & $\hat{\rho}_{\js}$ & 95\% CI \\
    		\midrule
    		\multirow{3}{*}{Cell A} & Sine & -178.5 & 0.13 & (0.01, 0.26) & 0.28 & (0.03, 0.53) \\
    		& Cosine & -176.3 & 0.18 & (0.04, 0.32) & 0.31 & (0.12, 0.50) \\
    		& Non-param. & -- & 0.15 & (0.03, 0.27) & 0.25 & (0.00, 0.50) \\
    		\midrule
    		\multirow{3}{*}{Cell B} & Sine & -132.6 & 0.45 & (0.30, 0.59) & 0.58 & (0.41, 0.74) \\
			& Cosine & -134.3 & 0.41 & (0.26, 0.57) & 0.51 & (0.34, 0.69) \\
			& Non-param. & -- & 0.45 & (0.24, 0.66) & 0.53 & (0.34, 0.72) \\    		
    		\bottomrule
    	\end{tabular}%
    	\caption{Estimates and approximate 95\% confidence intervals from maximum likelihood and non-parametric estimates of the Fisher-Lee and Jammalamadaka-Sarma correlation coefficients for the two cell datasets presented in Figure \ref{fig_realdata_scatter}. For the sine and cosine models, the maximum log-likelihood attained at the parameter estimates is shown in the `log-lik' column.} 
    	\label{tab:realdata_ests}%
    \end{table}%

	\begin{figure}[!hptb]
	\centering
	\subfloat[Example cell A: cosine model]%
	{\includegraphics[width = 0.5\linewidth]{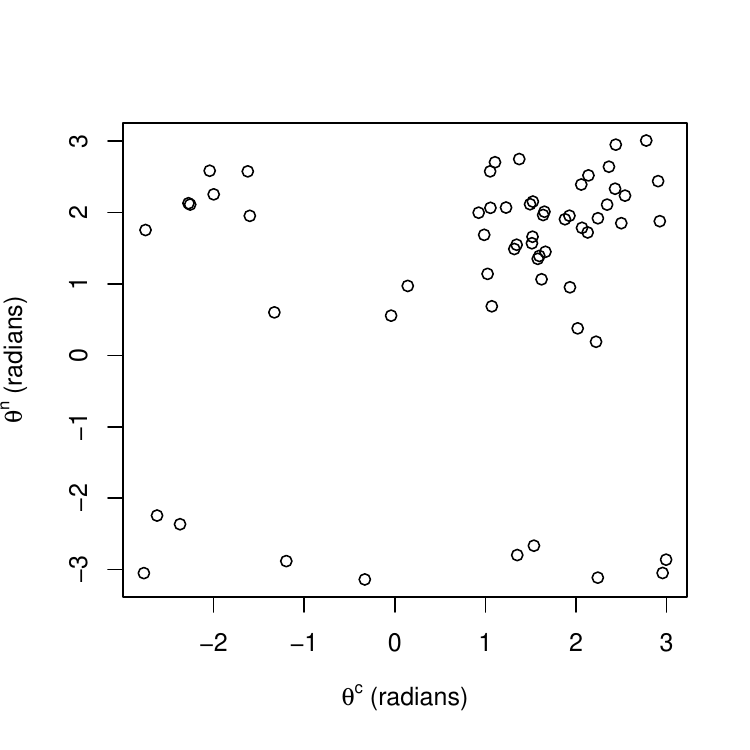}} 
	\hfill
	\subfloat[Example cell B: sine model]%
	{\includegraphics[width = 0.5\linewidth]{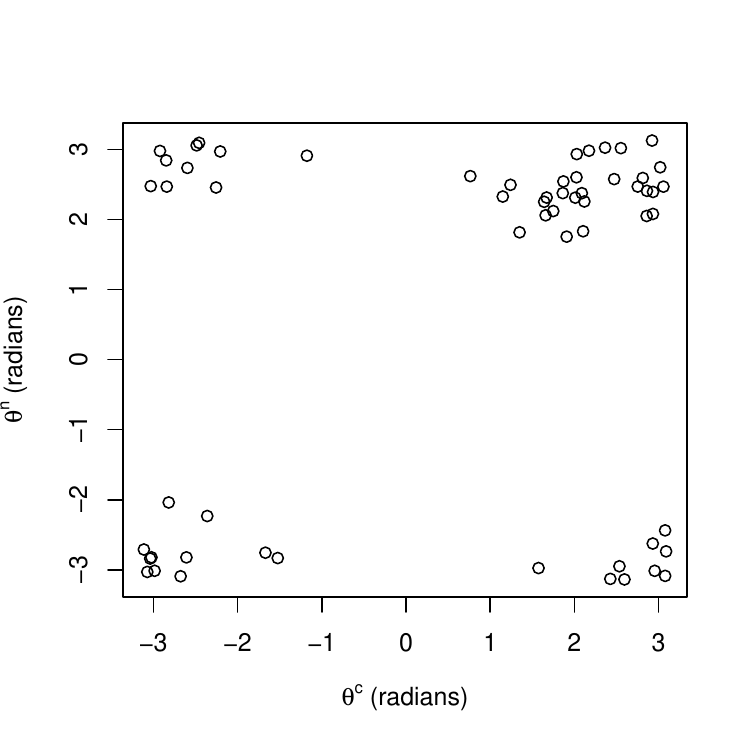}}
	\caption{Scatterplots of data simulated from the best-fit von Mises models for the two NIH 3T3 cells.} 
	\label{fig_fitted_scatter}	
\end{figure}

\section{Conclusion}
This paper studied the Jammalamadaka-Sarma and Fisher-Lee circular correlation coefficients for the von Mises sine and cosine models.  The theoretical properties of these coefficients were first considered;  then, their estimation from data via likelihood-based inference was presented.  The performance of the MLEs was illustrated via simulation studies and real data examples, and compared with the non-parametric counterparts of these circular correlation coefficients.  When bivariate angular data may be reasonably described by a von Mises sine or cosine model, then the MLE-based correlation coefficient estimates are more efficient.  Implementations of all computations are included in our R package \texttt{BAMBI}.

We note some directions for future work. First, both the simulation experiment and the real data analysis performed in this paper study the behaviors of the parametric correlation coefficient estimates when the respective von Mises models explain the data reasonably well. Thus, future research could investigate the effect of possible model misspecification on these parametric correlation coefficient estimates. On that note, goodness-of-fit tests for bivariate angular distributions, required to formally assess whether the von Mises models are reasonable fits to a given dataset, have not been addressed in the literature so far to the best of our knowledge.  Thus, a second possible avenue for future research could involve methodological investigation of formal goodness-of-fit tests for such bivariate angular data. 
	
	\begin{appendices}
	\section{Technical results required in the proofs of Theorem~\ref{thm_vms} and \ref{thm_vmc}}
	  \label{appen_tech_results}
		
		\begin{prop} \label{prop_vms_results}
			Let $(\Theta, \Phi) \sim \vms(\kappa_1, \kappa_2, \kappa_3, 0, 0)$. Then
			\begin{enumerate}[label = (\roman*)]
				\item \label{E_sin_x_sin_y} $E\left(\sin \Theta \sin \Phi \right) = \frac{1}{C_s} \frac{\partial C_s}{\partial \kappa_3}$.
				\item \label{E_sin_x_sin_y_sgn} $\sgn( E(\sin \Phi \sin \Theta)) = \sgn (\kappa_3)$.
				\item \label{E_cos_x_cos_y} $E\left(\cos \Theta \cos \Phi \right) = \frac{1}{C_s} \frac{\partial^2 C_s}{\partial \kappa_1 \partial \kappa_2}$.
				\item \label{E_cos_x} $E\left(\cos \Theta \right) = \frac{1}{C_s} \frac{\partial C_s}{\partial \kappa_1}$, and $E\left(\cos \Phi\right) = \frac{1}{C_s} \frac{\partial C_s}{\partial \kappa_2}$.
				\item \label{E_cos2_x} $E\left(\cos^2 \Theta \right) = \frac{1}{C_s} \frac{\partial^2 C_s}{\partial \kappa_1^2}$, and $E\left(\cos^2 \Phi\right) = \frac{1}{C_s} \frac{\partial^2 C_s}{\partial \kappa_2^2}$. 
				\item \label{E_sin_x_cos_y} $E(\sin \Phi \cos \Theta ) =  E(\sin \Theta \cos \Phi) = 0$.
				\item \label{E_sin_x_cos_x} $E(\sin \Theta \cos \Theta) = E(\sin \Phi \cos \Phi) = 0$.
			\end{enumerate}
			\label{vms_expec_res}
		\end{prop}
		
		\begin{proof}
			\begin{equation}\label{C_int}
			C_s = \int_{-\pi}^{\pi} \int_{-\pi}^{\pi} \exp\left(\kappa_1 \cos \theta + \kappa_2 \cos \phi + \kappa_3 \sin \theta \sin \phi \right)\: d\theta \: d\phi
			\end{equation}
			
			\noindent Because the integrand in (\ref{C_int}) is smooth and has continuous first and second order partial derivatives with respect to the parameters $(\kappa_1, \kappa_2, \kappa_3)$, and the limits of the integral are finite and constant (free of the parameters), partial differentiation with respect to the parameters, and the integration can be done in interchangeable orders (Leibniz's rule). 
			
			\begin{enumerate}[label=(\roman*)]
				\item Differentiating both sides of (\ref{C_int}) partially with respect to $\kappa_3$, and then applying Leibniz's rule, we get
				\begin{align*}
				\frac{\partial C_s}{\partial \kappa_3} &= \int_{-\pi}^{\pi} \int_{-\pi}^{\pi} \sin \theta \sin \phi \: \exp\left(\kappa_1 \cos \theta + \kappa_2 \cos \phi + \kappa_3 \sin \theta \sin \phi \right)\: d\theta \: d\phi \\
				&= C_s E\left(\sin \Theta \sin \Phi \right).
				\end{align*}
				
				\item Let $g(\lambda) =  \frac{\partial C_c}{\partial \lambda}$. Since $C_s > 0$, following part~\ref{E_sin_x_sin_y}, it is enough to show that $\sgn (g(\lambda)) = \sgn (\lambda)$. From the infinite series representation (\ref{del_C_lambda_expr}) we get 
				\begin{align*}
				g(\lambda) = 8 \pi^2  \sum_{m=1}^{\infty} m \binom{2m}{m} \frac{\lambda^{2m-1}}{(4\kappa_1 \kappa_2)^m} I_{m}(\kappa_1) I_{m}(\kappa_2) \lesseqgtr 0
				\end{align*}
				according as $\lambda \lesseqgtr 0$. This completes the proof.

				\item The result is obtained by partially differentiating  (\ref{C_int}) twice, once with respect $\kappa_1$ and then with respect to $\kappa_2$, and then by applying Leibniz's rule.
				
				\item The proof is given in \citet[Theorem~2(b)]{singh:2002}.
				
				\item The first half is obtained by partially  differentiating  (\ref{C_int}) twice  with respect to $\kappa_1$, and the second half, with respect to $\kappa_2$; followed by an application of Leibniz's rule.
				
				\item We shall only prove the first half. The proof of the second half is similar. It follows (see \citet{singh:2002}) that the conditional distribution of $\Phi$ given $\Theta = \theta$ is univariate von Mises $\vm\left(\kappa = a(\theta), \mu = b(\theta) \right)$, and the marginal density  of $\Theta$ is given by:
				\[
				f_\Theta (\theta) = \frac{2 \pi I_0(a(\theta))}{C_s} \exp(\kappa_1 \cos \theta) \one_{[-\pi, \pi)} (\theta)
				\]
				where 
				\[
				a(\theta) = \left\{ \kappa_2^2 + \kappa_3^2 \sin^2 \theta \right\}^{1/2} \text{ and } b(\theta) = \tan^{-1} \left(\frac{\kappa_3}{\kappa_2} \sin \theta \right).
				\]  
				Note that $f_\Theta$ is symmetric about $(\mu_1 = )\; 0$. Therefore, we have
				\begin{align*}
				E\left(\sin \Phi \cos \Theta \right) &= E \left[\cos \Theta \: E\left(\sin \Phi \mid \Theta \right)\right] \\
				&= E \left[\cos \Theta \: \frac{I_1(a(\Theta))}{I_0(a(\Theta)) } \: \sin(\beta(\Theta)) \right] \\
				&= E \left[\cos \Theta \: \frac{I_1(a(\Theta))}{I_0(a(\Theta)) } \: \frac{(\kappa_3/\kappa_2) \sin  \Theta}{\sqrt{1 + (\kappa_3/\kappa_2)^2 \sin^2 \Theta}} \right] \\
				&= 0,
				\end{align*}
				where the second equality follows from Proposition~\ref{prop_Esinx_vm}, and the last from the fact that the associated integral is an odd function.

				\item These results are immediate consequences of symmetry of the marginal distributions.   
			\end{enumerate}
		\end{proof}

		\begin{prop} \label{prop_Esinx_vm}
			Let $X$ have a univariate von Mises distribution $\vm(\kappa, \mu)$. Then $E(\sin X) = \frac{I_1(\kappa)}{I_0(\kappa)} \sin \mu$.
		\end{prop}
		\begin{proof}
			Because the density of $X$ is symmetric about $\mu$, we have,
			\begin{equation} \label{eqn_sinx}
			E[\sin(X - \mu)] = E(\sin X) \cos \mu - E(\cos X) \sin \mu = 0.
			\end{equation}
			Also (see, e.g., \citet[\S 9.6.19]{abramowitz:1964}),
			\begin{equation} \label{eqn_cosx}
			E[\cos(X - \mu)] = E(\cos X) \cos \mu + E(\sin X) \sin \mu = \frac{I_1(\kappa)}{I_0(\kappa)}.
			\end{equation}
			Solving for $E(\sin X)$ from (\ref{eqn_sinx}) and (\ref{eqn_cosx}) yields $E(\sin X) = \frac{I_1(\kappa)}{I_0(\kappa)} \sin \mu$.
		\end{proof}

		\begin{prop}  \label{prop_vmc_results}
			Let $(\Theta, \Phi) \sim \vmc(\kappa_1, \kappa_2, \kappa_3, 0, 0)$. Then
			\begin{enumerate}[label = (\roman*)]
				\item \label{E_c_cos_x_cos_y} $E\left(\cos \Theta \cos \Phi \right) = \frac{1}{C_c}  \frac{\partial^2 C_c}{\partial \kappa_1 \partial \kappa_2}$.
				\item \label{E_c_sin_x_sin_y} $E\left(\sin \Theta \sin \Phi \right) = \frac{1}{C_c} \left\{\frac{\partial C_c}{\partial \kappa_3} -  \frac{\partial^2 C_c}{\partial \kappa_1 \partial \kappa_2} \right\}$.		
				\item \label{E_c_sin_x_sin_y_sgn} $\sgn( E(\sin \Phi \sin \Theta)) = \sgn (\kappa_3)$.
				\item \label{E_c_cos_x} $E\left(\cos \Theta \right) = \frac{1}{C_s} \frac{\partial C_s}{\partial \kappa_1}$, and $E\left(\cos \Phi\right) = \frac{1}{C_s} \frac{\partial C_s}{\partial \kappa_2}$.
				\item \label{E_c_cos2_x} $E\left(\cos^2 \Theta \right) = \frac{1}{C_c} \frac{\partial^2 C_c}{\partial \kappa_1^2}$, and $E\left(\cos^2 \Phi\right) = \frac{1}{C_c} \frac{\partial^2 C_c}{\partial \kappa_2^2}$. 
				\item \label{E_c_sin_x_cos_y} $E(\sin \Phi \cos \Theta) = E(\sin \Theta \cos \Phi ) = 0$.
				\item \label{E_c_sin_x_cos_x} $E(\sin \Theta \cos \Theta) = E(\sin \Phi \cos \Phi) = 0$.
			\end{enumerate}
			\label{vmc_expec_res}
		\end{prop}
		
		\begin{proof}
			We have
			\begin{equation}\label{C_c_int}
			C_c = \int_{-\pi}^{\pi} \int_{-\pi}^{\pi} \exp\left(\kappa_1 \cos \theta + \kappa_2 \cos \phi + \kappa_3 \cos (\theta - \phi) \right)\: d\theta \: d\phi
			\end{equation}
			
			\noindent Using the same arguments as in the von Mises sine case, it follows that partial differentiation with respect to the parameters, and the integration can be done in interchangeable orders (Leibniz's rule). 
			
			\begin{enumerate}[label=(\roman*)]
				\item 	Differentiating both sides of  (\ref{C_c_int}) twice, once with respect $\kappa_1$ and then with respect to $\kappa_2$, and then by applying Leibniz's rule, we get
				\begin{align*}
				& \quad \frac{\partial^2 C_c}{\partial \kappa_1 \partial \kappa_2} \\
				&= \int_{-\pi}^{\pi} \int_{-\pi}^{\pi} \sin \theta \sin \phi \: \exp\left(\kappa_1 \cos \theta + \kappa_2 \cos \phi + \kappa_3 \cos (\theta - \phi) \right)\: d\theta \: d\phi \\
				&= C_c E \left(\cos \Theta \cos \Phi\right).
				\end{align*}
				
				\item Differentiating (\ref{C_c_int}) partially with respect to $\kappa_3$, and then applying Leibniz's rule, we get
				\begin{align*}
				\frac{\partial C_c}{\partial \kappa_3} &= \int_{-\pi}^{\pi} \int_{-\pi}^{\pi} \cos (\theta-\phi) \: \exp\left(\kappa_1 \cos \theta + \kappa_2 \cos \phi + \kappa_3 \cos (\theta - \phi) \right)\: d\theta \: d\phi \\
				&= C_c E\cos \left(\Theta - \Phi \right) = C_c E\cos \left(\cos \Theta \cos \Phi + \sin \Theta \sin \Phi \right).
				\end{align*}
				This, together with part~\ref{E_c_cos_x_cos_y} yields
				\[
				\frac{\partial C_c}{\partial \kappa_3} - 	\frac{\partial^2 C_c}{\partial \kappa_1 \partial \kappa_2} = C_c \: E (\sin \Theta \sin \Phi)
				\]
				
				\item Let $g(\kappa_3) =  \frac{\partial C_c}{\partial \kappa_3} - 	\frac{\partial^2 C_c}{\partial \kappa_1 \partial \kappa_2}$. Since $C_c > 0$, following part~\ref{E_c_sin_x_sin_y}, it is enough to show that $\sgn (g(\kappa_3)) = \sgn (\kappa_3)$. Straightforward algebra on the infinite series representations (\ref{del_C_c_k3_expr}) and (\ref{del_C_c_k1_k2_expr}) of $\frac{\partial C_c}{\partial \kappa_3}$ and $\frac{\partial^2 C_c}{\partial \kappa_1 \partial \kappa_2}$ yields,
				\begin{align*}
				& \quad g(\kappa_3) \\
				&= 2 \pi^2 \left\lbrace \sum_{m=1}^\infty I_{m-1}(\kappa_1) I_{m-1}(\kappa_2) I_m(\kappa_3) - \sum_{m=1}^\infty I_{m-1}(\kappa_1) I_{m+1}(\kappa_2) I_m(\kappa_3) \right. \\
				& \qquad  \left.  -  \sum_{m=1}^\infty I_{m+1}(\kappa_1) I_{m-1}(\kappa_2) I_m(\kappa_3) + \sum_{m=1}^\infty I_{m+1}(\kappa_1) I_{m+1}(\kappa_2) I_m(\kappa_3) \right\} \\
				&= 2 \pi^2 \sum_{m=1}^\infty [I_{m-1}(\kappa_1) -  I_{m+1}(\kappa_1)] [I_{m-1}(\kappa_2) - I_{m+1}(\kappa_2)] I_m(\kappa_3)  \\
				&=  \sum_{m=1}^\infty a_m \: I_m(\kappa_3) \numbereqn \label{del_C_c_k3_minus_k1k2}
				\end{align*} 
				where $a_m = 2 \pi^2 [I_{m-1}(\kappa_1) -  I_{m+1}(\kappa_1)] [I_{m-1}(\kappa_2) - I_{m+1}(\kappa_2)]$. Note that $(a_m)_{m \geq 1}$ is a decreasing sequence of positive real numbers since $I_n(x) > I_{n+1}(x)$ for $n \geq 1$ and $x \geq 0$. We consider the cases $\kappa_3 = 0$, $\kappa_3 > 0$ and $\kappa_3 < 0$ separately, and note the sign of $g(\kappa_3)$ in each case.
				
				\begin{enumerate}[label=(\alph*)]
					\item  If $\kappa_3 = 0$, then $I_m(\kappa_3) = 0$ for all $m = 1, 2, \cdots$. Consequently, the right hand side of (\ref{del_C_c_k3_minus_k1k2}) becomes zero.
					
					\item If $\kappa_3 > 0$, then $I_m(\kappa_3) > 0$ for all $m = 1, 2, \cdots$. Therefore, the right hand side of (\ref{del_C_c_k3_minus_k1k2}) is a series of positive terms, and hence is positive.
					
					\item If $\kappa_3 < 0$, then $I_m(\kappa_3) = (-1)^m I_m(|\kappa_3|)$ for $m = 1, 2, \cdots$, and the right hand side of (\ref{del_C_c_k3_minus_k1k2}) is an (absolutely convergent) alternating series
				   \[
				   S = \sum_{m = 1}^\infty (-1)^m \: a_m \: I_m(|\kappa_3|).
				   \]
				   Note that
				   \begin{align*}
				   	S &= - \sum_{m = 1}^\infty a_{2m - 1} \: I_{2m - 1}(|\kappa_3|) + \sum_{m = 1}^\infty a_{2m} \: I_{2m}(|\kappa_3|) \\
				   	&< - \sum_{m = 1}^\infty a_{2m - 1} \: I_{2m - 1}(|\kappa_3|) + \sum_{m = 1}^\infty a_{2m-1} \: I_{2m}(|\kappa_3|) \\
				   	&= - \sum_{m = 1}^\infty a_{2m - 1} \: [I_{2m - 1}(|\kappa_3|) - I_{2m}(|\kappa_3|)] = - S^* \\
				   	&< 0. 
				   \end{align*}
				   where the inequality in the second line follows from the fact that $(a_m)_{m \geq 1}$ is decreasing and positive, and that in the last line is a consequence of the fact that $S^*$, being a series of positive terms (since $I_{2m - 1}(|\kappa_3|) > I_{2m}(|\kappa_3|)$  for all $m \geq 1$), is positive.
				\end{enumerate}

				\item The first part is proved by partially  differentiating  (\ref{C_c_int})  with respect to $\kappa_1$, and the second part, with respect to $\kappa_2$; followed by an application of Leibniz's rule.
				
				\item The first half is obtained by partially  differentiating  (\ref{C_c_int}) twice  with respect to $\kappa_1$, and the second half, with respect to $\kappa_2$; followed by an application of Leibniz's rule.
				
				\item We shall only prove the first half. The proof of the second half is similar. It follows from \citet{mardia:2007} that the conditional distribution of $\Phi$ given $\Theta = \theta$ is univariate von Mies $\vm\left(\kappa = \kappa_{13}, \mu = \theta_0 \right)$, and the marginal density  of $\Theta$ is given by:
				\[
				g_\Theta (\theta) = \frac{2 \pi I_0(\kappa_{13}(\theta))}{C_c} \: \exp(\kappa_2 \cos \theta) \: \one_{[-\pi, \pi)} (\theta)
				\]
				where 
				\[
				\kappa_{13} (\theta) = \kappa_1^2 + \kappa_3^2 + 2 \kappa_1 \kappa_3 \cos \theta 
				\text{ and }	\theta_0 = \tan^{-1} \left(\frac{\kappa_3 \sin \theta}{\kappa_1 + \kappa_3 \cos \theta}  \right).
				\]  
				Note that $f_\Theta$ is symmetric about $(\mu_1 = )\; 0$. Therefore, we have
				\begin{align*}
				E\left(\sin \Phi \cos \Theta \right) &= E \left[\cos \Theta \: E\left(\sin \Phi \mid \Theta \right)\right] \\
				&= E \left[\cos \Theta \: \frac{I_1(\kappa_{13}(\Theta))}{I_0(\kappa_{13}(\Theta)) } \: \sin \tan^{-1} \left(\frac{\kappa_3 \sin \Theta}{\kappa_1 + \kappa_3 \cos \Theta}  \right) \right] \\
				&= E \left[\cos \Theta \: \frac{I_1(\kappa_{13}(\Theta))}{I_0(\kappa_{13}(\Theta)) } \: \frac{\left(\frac{\kappa_3 \sin \Theta}{\kappa_1 + \kappa_3 \cos \Theta}  \right)}{\sqrt{1 + \left(\frac{\kappa_3 \sin \Theta}{\kappa_1 + \kappa_3 \cos \Theta}  \right)^2}} \right] \\
				&= 0,
				\end{align*}
				where the second equality follows from Proposition~\ref{prop_Esinx_vm}, and the last from the fact that the associated integral is an odd function. 

				\item These results are immediate consequences of symmetry of the marginal distributions.    
			\end{enumerate}
		\end{proof}

	\end{appendices}
	
	\bibliographystyle{apa}
	\bibliography{circ_cor_ref}
	
\end{document}